\documentclass{amsart}

\usepackage{amsmath, amssymb, amsthm, amsfonts}

\input xy
\xyoption{all}

\usepackage{tikz}
\usepackage{dynkin-diagrams}
\usepackage{array}   
\newcolumntype{C}{>{$}c<{$}} 
\usepackage{diagbox}
\usepackage{enumerate}

\usepackage{hyperref}

\numberwithin{equation}{section}

\theoremstyle{plain}
\newtheorem{theorem}[subsubsection]{Theorem}
\newtheorem{lemma}[subsubsection]{Lemma}

\newtheorem{prop}[subsubsection]{Proposition}
\newtheorem{cor}[subsubsection]{Corollary}

\theoremstyle{definition}
\newtheorem{defn}[subsubsection]{Definition}

\newtheorem{remark}[subsubsection]{Remark}

\def\AA{\mathbb{A}}

\def\CC{\mathbb{C}}

\def\HH{\mathbb{H}}

\def\PP{\mathbb{P}}
\def\QQ{\mathbb{Q}}

\def\SS{\mathbb{S}}

\def\WW{\mathbb{W}}

\def\ZZ{\mathbb{Z}}

\newcommand\cB{\mathcal{B}}
\newcommand\cC{\mathcal{C}}

\newcommand\cE{\mathcal{E}}
\newcommand\cF{\mathcal{F}}
\newcommand\cG{\mathcal{G}}

\newcommand\cI{\mathcal{I}}

\newcommand\cL{\mathcal{L}}
\newcommand\cM{\mathcal{M}}

\newcommand\cO{\mathcal{O}}
\newcommand\cP{\mathcal{P}}

\newcommand\cS{\mathcal{S}}
\newcommand\cT{\mathcal{T}}

\newcommand\cV{\mathcal{V}}

\def\ba{\mathbf{a}}

\def\bd{\mathbf{d}}

\newcommand\frb{\mathfrak{b}}

\newcommand\frg{\mathfrak{g}}

\newcommand\frp{\mathfrak{p}}

\newcommand\frt{\mathfrak{t}}

\newcommand\dt{\widehat{\mathfrak{t}}}
\newcommand\wdS{\widehat{\Sigma}}

\newcommand{\gr}{\textup{gr}}

\newcommand\id{\textup{id}}

\newcommand\Loc{\textup{Loc}}

\newcommand{\opp}{\textup{opp}}

\newcommand\pt{\textup{pt}}

\newcommand\Rep{\textup{Rep}}

\newcommand\rk{\textup{rk}}

\newcommand\Spec{\textup{Spec}\ }

\newcommand\Stab{\textup{Stab}}

\newcommand\triv{\textup{triv}}

\newcommand\Aut{\textup{Aut}}
\newcommand\Hom{\textup{Hom}}

\newcommand\Isom{\textup{Isom}}

\newcommand\GL{\textup{GL}}

\newcommand\PGL{\textup{PGL}}

\newcommand\Br{\textup{Br}}

\newcommand{\Ad}{\textup{Ad}}
\renewcommand\sc{\textup{sc}}

\newcommand{\defeq}{\vcentcolon=}

\newcommand{\To}[1]{\overset{#1}{\longrightarrow}}

\newcommand\quash[1]{}

\renewcommand\a\alpha
\renewcommand\b\beta
\newcommand\G\Gamma
\newcommand\g\gamma
\renewcommand\d\delta
\newcommand\D\Delta

\renewcommand{\r}{\rho}
\newcommand{\s}{\sigma}

\renewcommand{\l}{\lambda}
\renewcommand{\L}{\Lambda}

\newcommand\Conn{\textup{Conn}}

\newcommand{\gt}{\frg(\!(t)\!)}

\newcommand{\SFL}{\textup{SFLoc}}
\newcommand{\SGL}{\textup{SGrLoc}}

\newcommand{\Fun}{\textup{Fun}}
\newcommand{\Dem}{\textup{Dem}}

\title[Stokes phenomenon of Kloosterman and Airy connections]{Stokes phenomenon of Kloosterman\\ and Airy connections}

\author{Andreas Hohl}
\thanks{A.H. was supported by the Deutsche Forschungsgemeinschaft (DFG, German Research Foundation), Projektnummer 465657531, and the grant G0B3123N from the Fonds voor Wetenschappelijk
Onderzoek – Vlaanderen (FWO, Research Foundation – Flanders).}
\address{(A.H.) KU Leuven, Departement Wiskunde, Celestijnenlaan 200B, B-3001 Leuven, Belgium; \textit{Current address:} Technische Universität Chemnitz, Fakultät für Mathematik, 09107 Chemnitz, Germany}
\author{Konstantin Jakob}
\thanks{K.J. acknowledges support (through Timo Richarz) by the European Research Council (ERC) under Horizon Europe (grant agreement no 101040935), by the Deutsche Forschungsgemeinschaft (DFG, German Research Foundation) TRR 326 \textit{Geometry and Arithmetic of Uniformized Structures}, project number 444845124 and the LOEWE professorship in Algebra, project number LOEWE/4b//519/05/01.002(0004)/87}
\address{(K.J.) TU Darmstadt, Fachbereich Mathematik, Schlossgartenstraße 7, 64289 Darmstadt, Germany}
\date{}

\begin{document}

\begin{abstract}
    We define categories of Stokes-filtered and Stokes-graded $G$-local systems for reductive groups $G$ and use the formalism of Tannakian categories to show that they are equivalent to the category of $G$-connections. We then use the interpretation of moduli spaces of Stokes-filtered $G$-local systems as braid varieties to prove physical rigidity of two well-known families of cohomologically rigid connections, the Kloosterman and Airy connections. In the Kloosterman case, our proof relies on Steinberg's cross-section. 
\end{abstract}

\maketitle
\tableofcontents

\section{Introduction}
\subsection{Kloosterman connections} Deligne's Kloosterman sheaves, defined in \cite{Deligne2}, are sheaf-theoretic versions of classical Kloosterman sums. They are $\ell$-adic sheaves on $\PP^1\setminus \{0,\infty\}$, with tame ramification at $0$, and wild ramification at $\infty$. Their de Rham analogues are special cases of irregular hypergeometric connections. These connections (as well as Deligne's Kloosterman sheaves) are physically rigid in the sense that they are uniquely determined by the isomorphism classes of their formal restrictions to $0$ and $\infty$.

In \cite{HNY}, analogues of Kloosterman sheaves for a reductive group $G$ were constructed using methods from the geometric Langlands program. They have many properties which are natural generalizations of properties of Deligne's Kloosterman sheaves: 

\begin{itemize}
    \item Their slope at $\infty$ is $1/h$, where $h$ is the Coxeter number of $G$ (assuming $G$ is almost simple).
    \item The monodromy at $0$ is regular. 
    \item They are cohomologically rigid. 
\end{itemize}

When $G=\GL_n$, it was shown by Katz \cite{Katz} and Bloch--Esnault \cite{BE} that for irreducible connections the notion of cohomological and physical rigidity are equivalent. For other reductive groups, no general criteria are known to decide when cohomological rigidity coincides with physical rigidity. In fact it is known that the two notions disagree even on $\PP^1$. Already in the tame case, in \cite{KNP} the authors give an example of a cohomologically rigid $\PGL_{2}$-local system on $\PP^1\setminus \{0,1,\infty\}$ which is not physically rigid. 

In \cite[Conjecture 7.1]{HNY} it was predicted that Kloosterman sheaves for reductive groups are physically rigid. The de Rham analogue of this conjecture was recently proven in \cite{Yi} when $G$ is of adjoint type and the monodromy at $0$ is regular unipotent. The proof is based on methods from \cite{Zhu}, where it is shown that the de Rham analogues of Kloosterman sheaves coincide with the rigid irregular connection constructed by Frenkel--Gross \cite{FG}. 

More precisely, \cite{Yi} uses ideas from the geometric Langlands program. Given a $G$-connection on $\PP^1\setminus \{0,\infty\}$ which is isomorphic to Frenkel--Gross's connection formally around $0$ and $\infty$, it is shown that it is the eigenvalue of a Hecke eigensheaf which coincides with the eigensheaf for the Frenkel--Gross connection. The construction of this eigensheaf uses a version of Beilinson--Drinfeld localization and the existence of generic oper structures.

In this paper, we prove the following more general version. 

\begin{theorem}[{see Theorem \ref{thm:KloostermanPhysicalRigidity}}]\label{thm:main}
    Assume $G$ is a simple complex algebraic group and let $\cC$ be a regular conjugacy class in $G$. Moreover, assume $G$ is simply connected or $\cC$ is the regular unipotent class. Then, the Kloosterman $G$-connection with local monodromy at $0$ in $\cC$ is physically rigid.
\end{theorem}

Our methods are more elementary and completely different. We make use of the Stokes phenomenon for irregular connections. The Stokes phenomenon was first observed by Stokes in his study of Airy's equation. Roughly speaking, the asymptotic behaviour of local sectorial solutions to the Airy equation jumps when crossing certain directions around $\infty$. The jumps are described by unipotent linear transformations, nowadays known as Stokes matrices. It is fruitful to think of these data as generalized monodromy data, and they allow us to understand irregular connections in purely topological terms (as classical monodromy does for regular singular connections). 

There are several ways of encoding the Stokes phenomenon topologically, all of which yield complete invariants for meromorphic connections on Riemann surfaces:

\begin{itemize}
    \item Stokes local systems, which essentially record Stokes matrices, 
    \item Stokes-graded local systems, and
    \item Stokes-filtered local systems.
\end{itemize}
Therefore, to show physical rigidity of a connection, one may instead work with either of the above objects.

The concept of a Stokes local system was first formulated by Boalch in \cite{Boalch01} in the case where $G=\GL_n$, and it has been extended to more general groups in \cite{BoalchGbun, BoalchGeomBraiding, BY}. In fact, Stokes local systems are actual local systems on a modified Riemann surface (with boundary), and hence they can be described very explicitly in terms of ``wild monodromy representations''. Moduli spaces of Stokes local systems are called ``wild character varieties’’, yielding an elegant generalization of the classical character varieties in the framework of connections with irregular singularities.
These spaces carry a rich geometry and we refer to the literature mentioned above for full details.

On the other hand, the concept of a Stokes filtration (in the case $G=\GL_n$) was set up by Deligne and Malgrange (see \cite{Mal} for an exposition). The notion of Stokes-graded local system has been studied in \cite{BoalchTop} as an intermediate step between Stokes filtrations and Stokes local systems when $G=\GL_n$.

For $G=\GL_n$, these three approaches and their equivalence are discussed in detail in \cite{BoalchTop}. The generalization to arbitrary reductive groups should be known to experts, but is not readily available in the literature. We therefore extend part of \textit{loc.~cit.}\ by introducing categories of Stokes-filtered and Stokes-graded $G$-local systems and proving that they are equivalent to the category of $G$-connections. We take advantage of the Tannakian structure of the category of meromorphic connections on a Riemann surface to do so. The Tannakian formalism identifies this category with the category of finite dimensional complex representations of the so called \textit{wild fundamental group}, as defined by Martinet--Ramis in \cite{MR}. Stokes local systems arise as representations of the wild fundamental group. 

In the Tannakian situation there is a natural way to define ``principal'' objects for a reductive group $G$, namely one considers the category of exact tensor functors from $\Rep(G)$ into the Tannakian category in question. We give an intrinsic description of Stokes-filtered and Stokes-graded $G$-local systems, and prove that they are equivalent to the category of exact tensor functors from $\Rep(G)$ to the categories of Stokes-filtered and Stokes-graded local systems respectively. This reduces the equivalence of Stokes-filtered and graded $G$-local systems to the classical case discussed in \cite{BoalchTop}. 

We also briefly discuss Stokes $G$-local systems and their equivalence to the other notions. The fact that they classify $G$-connections is already shown in \cite[Appendix~A]{BoalchGeomBraiding} in the unramified case. The equivalence between Stokes $G$-local systems and $G$-connections in the general situation was also proven recently in \cite{HS} independently, without appealing to Stokes filtrations or gradings. 

To prove Theorem \ref{thm:main}, we take the point of view of Stokes-filtered local systems. We show that -- in our cases of interest -- their moduli spaces can be identified with quotients of braid varieties, in this form first considered in \cite{BBMY}, but originally dating back to the work of Deligne \cite{Deligne}. These braid varieties are therefore stacky variants of some special cases of the wild character varieties of \cite{BY}. This presentation of certain wild character varieties seems to be known to experts, but again is not fully documented in the literature. The translation from braid matrices to Stokes matrices is explained in detail in the $\GL_2$-case in Boalch's paper \cite{BoalchCalabi}, where wild character varieties are related to Sibuya's work \cite{Sibuya}.

The $G$-connections we work with satisfy a special condition (called \emph{isoclinic}) at their irregular singularity.  Roughly speaking, locally formally they are diagonalizable after a suitable pullback, with regular semisimple leading term. In this case, all non-zero slopes appearing in the adjoint representation are the same. The Deligne--Simpson problem for irregular connections on $\PP^1\setminus \{0,\infty\}$ with an isoclinic irregular singularity at $\infty$ and a regular singularity at $0$ has been studied in \cite{JY}. In addition, the cohomologically rigid cases are classified there. For $G=\GL_n$ and slope $d/n$, this work was preceded by \cite{KLMNS}.  

When the denominator of the slope $\nu = d/m$ (written in lowest terms) is a regular elliptic number for the group $G$, the corresponding braid is easy to describe. Up to a cyclic shift of words, it is given as $\tilde{w}^d$ for the positive braid lift $\tilde{w}$ of a regular elliptic element $w$ of order $m$ of minimal length in its conjugacy class, cf.\ \cite[\S 4.4]{BBMY}.
In the case of Kloosterman connections, this braid is simply the lift of a Coxeter element, and these braid varieties are then essentially identified with Steinberg's cross-section. The Steinberg section is a certain naturally defined subset of a reductive group, which in the simply connected case intersects every regular class in a unique point. It can be thought of as a group version of Kostant's section in Lie theory. Such ideas have already been used in the context of Stokes phenomena in works of Guest and Ho \cite{GH}, but we bypass many of their computations in our case by using the simple description of the attached braid. 

Fixing the local monodromy of a Kloosterman connection to lie in a certain regular class amounts to intersecting Steinberg's cross-section with this class, proving physical rigidity (and giving yet another construction of Kloosterman connections). 

\subsection{Generalized Kloosterman connections}
In \cite{Yun}, Yun defines an $\ell$-adic generalization of Kloosterman sheaves for any simple group $G$. They are again constructed using methods from the geometric Langlands program. This construction goes through in the de Rham setting as well and procudes an integrable connection with properties generalizing those of Kloosterman connections. The input in the construction is a regular elliptic number $m$, which is the order of a regular elliptic element $w\in W$. To any conjugacy class in the Weyl group, Lusztig has attached a unipotent class in $G$, see \cite{Lu}. We denote by $\cC_{w}$ the unipotent class attached to the regular elliptic element $w$. Generalized Kloosterman connections are algebraic connections on $\PP^1\setminus \{0,\infty\}$ which are irregular at $\infty$ with slope $1/m$ at $\infty$, and regular singular at $0$ with monodromy in the class $\cC_{w}$. Using results of He and Lusztig in \cite{HL} on generalized Steinberg sections associated to elliptic elements, we prove the following. 
\begin{theorem}[see Corollary \ref{cor:genkloost}]
Generalized Kloosterman connections are physically rigid. 
\end{theorem}
A related construction is the notion of $\theta$-connections, which are defined in \cite{Chen}, following an idea of Yun. More recently, under some mild assumptions, and for adjoint groups, Chen and Yi have proven that $\theta$-connections and generalized Kloosterman connections agree, using methods from geometric Langlands \cite{CY}. Under their assumptions, they also deduce physical rigidity. As for usual Kloosterman connections, our method gives an alternate proof, valid for any simple group $G$, and without any assumptions other than fixing the slope at $\infty$ to be $1/m$ and fixing the conjugacy class of local monodromy at $0$ to be $\cC_{w}$.  

\subsection{Airy connections} In addition to physical rigidity of Kloosterman connections, we study another well-known family of connections which generalize the classical Airy equation $y''=zy$ on $\CC$. 

For $G=\GL_n$, the $\ell$-adic analogues of Airy's equation were first studied in \cite{Katz}. In \cite{KS2}, they were uniformly defined for any simple group $G$, and more general instances are constructed in \cite{JKY}, following the method of \cite{HNY}. The most general construction can be found in \cite{JY}. Airy $G$-connections are algebraic $G$-connections on $\AA^1$ with slope $1+1/h$ at $\infty$. In this paper, we prove the following property. 

\begin{theorem}[{see Theorem \ref{thm:AiryPhysicalRigidity}}]
Airy $G$-connections are physically rigid. 
\end{theorem} 

The proof again relies on the identification of moduli spaces of Stokes-filtered local systems with (quotients of) braid varieties. We remark that the methods of \cite{Yi} cannot be applied in this situation. The issue is that the corresponding moduli space of bundles with level structure on which the Hecke eigensheaf is defined is not as well-behaved as in the Kloosterman case.

Our proof uses mostly basic properties of braid varieties. In the Airy case they are smooth, irreducible, and have explicit dimension equal to the rank of the group. The moduli space of Airy connections is then identified with the quotient of a braid variety by the natural action of a maximal torus which acts transitively with finite stabilizer, proving physical rigidity.

\subsection{Organization of the paper} In Section \ref{s:setting}, we introduce the basic setting and notation. In Section \ref{s:stokes}, we recall basic facts about the Stokes phenomenon and introduce Stokes filtrations and Stokes gradings for general reductive groups $G$. We prove the equivalence of Stokes-filtered and Stokes-graded $G$-local systems with $G$-connections. In Section \ref{s:braids}, we introduce the moduli stack of Stokes-filtered $G$-local systems and relate it to braid varieties in special cases.
In Section \ref{s:rigidity}, we recall the definition of Kloosterman connections, basic facts about the Steinberg section, and we prove physical rigidity of Kloosterman $G$-connections, as well as physical rigidity for generalized Kloosterman connections. We then go on to defining Airy $G$-connections and proving their physical rigidity. 

\subsection*{Acknowledgements}

We thank Philip Boalch for helpful comments on an earlier version of this paper, and in particular for pointing out an alternative approach to the rigidity of Kloosterman connections, described in Remark~\ref{rem:RigidityWCV}. KJ thanks Ian Le for answering questions on braid varieties, and Pengfei Huang, and Zhiwei Yun for helpful discussions.

\section{Setting and notation}\label{s:setting}
\subsection{$G$-local systems}
Let $X$ be a topological space, and let $\cG$ be a sheaf of groups on $X$. By a $\cG$-torsor on $X$ we mean a sheaf of sets $\cF$ on $X$ with a (right) action $\cF \times \cG \to \cF$ such that 
\begin{itemize}
    \item[(1)] there exists an open covering $X=\bigcup_{i\in I} U_i$ such that $\cF(U_i)\neq \emptyset$ for each $i\in I$, and
    \item[(2)] for each open $U\subseteq X$ with $\cF(U)\neq\emptyset$, the action $\cF(U)\times\cG(U)\to\cG(U)$ is free and transitive.
\end{itemize}

\begin{defn}
    Let $G$ be a group. A $G$-local system $\cF$ on $X$ is a torsor for the constant sheaf $G$ on $X$. 
\end{defn}
Passing to the espace étalé $E_{\cF}$ of $\cF$, this is the same as a local homeomorphism $\pi \colon E_{\cF} \to X$, every fiber of which is a principal homogeneous $G$-space (it carries a free and transitive $G$-action). From now on we will use both viewpoints without changing notation.

From now on, we let $G$ be a connected complex reductive algebraic group, and we denote by $\Rep(G)$ the category of complex algebraic representations of $G$. Given any $G$-local system $\cF$ on $X$, and any representation $(\pi,V)$ of $G$, we can form the associated local system 
\[\cF \times^{G} V = (\cF \times V)/G\]
where the $G$-action on $\cF\times V$ is given by $(x,v).g = (xg, \pi(g^{-1})v)$. This is a local system of complex vector spaces on $X$ with stalk $V$. We sometimes denote the associated local system by $\pi(\cF)$ or $\cF^{V}$. 

Let $X$ be a path-connected, locally path-connected and semi-locally simply connected topological space. We will sometimes refer to this as a well-behaved topological space. Recall the following facts about $X$.

The category $\Loc(X)$ of local systems of complex vector spaces on $X$ is a neutral Tannakian category (cf. \cite{DM} for details on this notion). The choice of any point $x \in X$ defines a fiber functor by taking the stalk. Denote by $\pi_1^{\otimes}(X,x)$ the Tannakian fundamental group of $X$. It is naturally identified with the pro-algebraic completion of the usual fundamental group $\pi_1(X,x)$. 

Moreover, for a complex reductive group $G$, by \cite[\S 6]{Sim1} the groupoid of $G$-local systems can be identified with the groupoid of exact faithful tensor functors 
\[ \Rep(G) \to \Loc(X), \]
where $\Rep(G)$ denotes the category of complex finite dimensional algebraic representations of $G$.
For a $G$-local system $\cF$ on $X$, the corresponding tensor functor is given by the associated local system construction.

\subsection{The group $G$ and related data}
From here on, until otherwise mentioned, we will fix the following notation and data. Let $G$ be a complex reductive group. We fix a maximal torus $T$ in a Borel subgroup $B$ of $G$ throughout. The corresponding Lie algebras will be denoted by $\frt\subset \frb\subset \frg$, so that $\frt$ is a Cartan subalgebra of $\frg$, and $\frb \subset \frg$ is a Borel subalgebra. Then, we have the Weyl group $W=W_T=N_{G}(T)$, defined with respect to $T$.

We denote by $\cB$ the variety of Borel subgroups of $G$ (or the variety of Borel subalgebras of $\frg$). It is isomorphic to the flag variety $G/B$. Attached to $G$ we have the universal Weyl group $(\WW,S)$ with simple reflections $S$. The group $G$ acts diagonally on $\cB \times \cB$ via its adjoint action, and we have a natural bijection
\begin{align}
    \WW \overset{\textup{1:1}}{\longleftrightarrow}\{G\textup{-orbits on } \cB \times \cB \}.
\end{align}
For $w\in \WW$, we denote the corresponding $G$-orbit by $\cO(w)$. Two Borel subgroups $B_1,B_2\subset G$ are said to be in relative position $w$ if $(B_1,B_2)\in \cO(w)$. In this case, we will sometimes write $B_1 \xrightarrow{w} B_2$. Concretely, our choice of $T\subset B$ identifies $W_T \cong \WW$ and $\cB \cong G/B$, and under this identification, $\cO(w)$ is the orbit of $(B,wB)$ in $G/B\times G/B$. Then, $xB \xrightarrow{w} yB$ if and only if $x^{-1}y \in BwB$.

Recall that a standard parabolic subgroup of $G$ is a parabolic subgroup containing the fixed Borel subgroup $B$. Each conjugacy class of parabolic subgroups contains a unique standard parabolic. They are in bijection with proper subsets $J$ of the positive roots, therefore each standard parabolic subgroup determines a parabolic subgroup $\WW_{J}$ in the abstract Weyl group $\WW$. Let $P_{J},P_{J'}$ be two standard parabolic subgroups, and denote by $\cP_{J}$ (resp.\ $\cP_{J'}$) the varieties of parabolic subgroups conjugate to $P_J$ (resp.\ $P_{J'}$). The set of diagonal $G$ orbits on $\cP_{J} \times \cP_{J'}$ is in bijection with $ \WW_{J} \backslash \WW / \WW_{J'}$. 

\subsection{The base curve}
We denote by $\Sigma$ a smooth complex algebraic curve. For a finite set of points $\ba \subset \Sigma$ we denote by $\pi\colon \wdS \to \Sigma$ the real oriented blow-up of $\Sigma$. We let 
\[\partial=\pi^{-1}(\ba) = \bigcup_{a\in \ba} \partial_{a}. \]
Sometimes we will assume $\ba$ to be a singleton to simplify notation. The general case is easily deduced.

\section{Stokes filtrations and Stokes gradings}\label{s:stokes}
Most of what we cover in this section is well-known to experts, but some parts are not readily available in the literature. We include this section for clarification and convenience of the reader.

In particular, the general framework for Stokes data over a complex reductive group $G$ (as presented in \S 3.1 and \S 3.2) has been set up in \cite{BY}.

\subsection{Exponential local system and irregular classes}
Assume $\ba=\{a\}$ so that $\partial$ is a single circle. Let $z$ be a local coordinate vanishing at $a$. The exponential local system $\cI\to \partial$ is the local system of sets whose local sections are (Germs on small sectors of) rational functions of the form
\[ q = \sum_{i=1}^{r} a_i z^{-k_{i}}\]
for some $r\in\ZZ_{>0}$, some $k_{i}\in \QQ_{>0}$ and $a_{i}\in \CC$ and for some fixed choice of a branch of $\log(z)$. We refer to \cite[\S 5.1]{BoalchTop} for details. We denote by $\cT$ the local system of pro-tori whose character group at each direction $d\in \partial$ is the free abelian group $\cI_{d}$. 
\begin{defn}[{\cite[Definition~5]{BY}}]
    An $\cI$-graded $G$-local system on $\partial$ is a $G$-local system $\cF$ together with a morphism of local systems of groups $\cT\to \Aut(\cF)$ factoring over an algebraic quotient of $\cT$. 
\end{defn}

Recall that two $\cI$-graded $G$-local systems on $\partial$ are said to have the same irregular class if they are locally isomorphic as graded local systems on $\partial$. By \cite[\S 3.5]{BY}, an irregular class $\Theta$ is given by a quotient $T_d$ of $\cT_{d}$ for any $d\in \partial$, and a conjugacy class of embeddings $T_d\hookrightarrow G$.

\begin{lemma}
    An irregular class is determined by a conjugacy class of morphisms $\cT_{d}\to G$.
\end{lemma}
\begin{proof}
Given two maps $j,j'\colon\cT_{d} \to G$ which are conjugate, say $j'=gjg^{-1}$ for some $g\in G$, we denote by $T_d$ and $T'_{d}$ the images of $\cT_d$ under $j$ and $j'$, respectively. The map $\l \mapsto {}^{g}\l$ with ${}^{g}\l(t') = \l(g^{-1}t'g)$ gives an isomorphism $X^*(T_d) \to X^*(T_d')$. It is easy to check that the embeddings $j^*\colon X^*(T_d) \to X^*(\cT_{d})$ and $(j')^*\colon X^*(T'_d) \to X^*(\cT_{d})$ have the same image, hence determine a well-defined algebraic quotient of $\cT_{d}$. 
\end{proof}

We say that two parabolic subgroups of $G$ are associates (or associated) if their Levi factors are conjugate. By an associate class we mean a conjugacy class of Levi subgroups of $G$. Given an irregular class $\Theta$, the centralizer of the image of any representative $\cT_d \to G$ is a Levi subgroup. In this way, an irregular class $\Theta$ determines an associate class $\bd(\Theta)$ for $G$, such that the center of each Levi subgroup in $\bd(\Theta)$ is conjugate to $T_d$. 
\begin{defn}
    We say that an irregular class $\cT_{d}\to G$ is \emph{toral} if the associate class $\bd(\Theta)$ is the conjugacy class of maximal tori. 
\end{defn}

\begin{remark}
    For $G=\GL_n\defeq \GL_n(\CC)$, we can make this more explicit. The irregular class is equivalent to the dimension vector determined by $\Theta$, i.e.\ a map $\pi_0(\cI)\to \ZZ_{\ge 0}$, by \cite[Proposition 8]{BY}.
    \end{remark}

\subsection{Stokes directions, singular directions, and the exponential dominance order}\label{sec:stokesdir}

Let $p=[q]\in \cI$ be a point on the exponential covering space (the étalé space of the exponential local system), i.e.\ a germ of a function $q(z)=\sum_{i=1}^r a_i z^{-k_i}$ (for some choice of branch of $\log z$, $r\in \ZZ_{>0}$, $a_i\in \CC$ and $k_i\in\QQ_{>0}$) on a small sector around the direction $d=\pi([q])\in \partial$. Then $p$ is \textit{oscillatory} if and only if $\textup{Re}(q(z)) = \textup{Re}(q(re^{i\theta})) \to 0$ for $\theta=d$ and $r\to 0$. On the other hand, $p$ is \textit{a point of maximal decay}, if $\exp(q)$ has maximal decay as $z\to 0$ along $d$. Equivalently, the monomial in $q$ of highest pole order is real and negative as $z\to 0$ along $\theta=d$. 

Fix an irregular class $\Theta$ for now. For any $d\in \partial$ let $T_d$ be the algebraic quotient of $\cT_{d}$ determined by $\Theta$, with corresponding finite rank sublattice $I_d\subset \cI_{d}$.

Choose an embedding $T_d \hookrightarrow G$ corresponding to the irregular class $\Theta$, well-defined up to conjugacy. Then, for any $d\in\partial$, write $\frg = \bigoplus_{\lambda \in I_d} \frg_{\lambda}$. While this decomposition itself is not well-defined, the set 
\[\Lambda_{d} = \{ \lambda \in I_{d} \mid \frg_{\l} \neq 0 \}\]
only depends on the conjugacy class of $T_d \hookrightarrow G$. 
\begin{defn}The set of \textit{Stokes directions} $\SS=\SS_{\Theta}$ of $\Theta$ is the set of $d \in \partial$ such that there is a non-zero oscillatory $\lambda$ with $\frg_{\lambda} \neq 0$.

The set of \textit{singular directions} (or anti-Stokes directions) $\AA=\AA_{\Theta}$ of $\Theta$ is the set of $d \in \partial$ such that there is a non-zero $\lambda$ which is a point of maximal decay with $\frg_{\lambda} \neq 0$.
\end{defn}

The set of Stokes directions and the set of singular directionds depends only on the irregular class $\Theta$.

\subsection{Positive and negative weights of $\Theta$} \label{ss:posnegweights} Recall that we fixed an irregular class $\Theta$. Each fiber of $\cI$ carries a partial ordering by exponential dominance. Explicitly, for $d\in \partial$ and $q_1,q_2 \in \cI_{d}$ we write $q_1 \le_{d} q_2$ if and only if $q_1=q_2$ or $\exp(q_1-q_2)$ has zero asymptotic expansion on an open sector around $d$.  In other words, if $q_1 \neq q_2$, the monomial of highest pole order of $q_1-q_2$ has negative real part along $d$.

By definition of the Stokes directions,  for any $d\in \partial \setminus \SS$, the exponential dominance ordering  induces a disjoint decomposition of the non-zero weights 
\[ \Lambda_{d} \defeq \{\l\in I_{d}\setminus \{0\} \mid \frg_{\l}\neq 0 \}
\]
into a positive and negative part $\Lambda_{d}\setminus \{0\} = \Lambda_{d}^+ \sqcup \Lambda_{d}^{-}$, determined by whether the real part of the monomial of highest pole order of $\lambda$ is positive or negative along $d$. Note that both parts of the decomposition are exchanged by $\l \mapsto -\l$.

Moreover, this decomposition satisfies the following two conditions:
\begin{enumerate} 
    \item For any $\l \in \Lambda_{d}$, either $\l\in \Lambda_{d}^{+}$ or $-\l\in \Lambda_{d}^{+}$.
    \item If $\l,\mu \in \Lambda_{d}^{+}$ and $\l+\mu \in \Lambda_{d}$, then $\l+\mu \in \Lambda_{d}^{+}$.
\end{enumerate}

Assume $d \in \AA$ is a singular direction. There is a partial ordering $\prec_d$ on $\L_{d}=\L_{\Theta,d}$ defined by $q_1 \prec_d q_2$ if and only if $q_1-q_2$ is a point of maximal decay. It is clear that $\prec_d$ is refined by exponential dominance. 

\begin{remark}
    Our notation differs slightly from the notation in \cite{BoalchTop}.
    For $G=\GL_n$, Boalch denotes by $I_d$ the set of active exponents, which is finite. These are the weights in $\cI_{d}$ that appear in the standard representation. Outside $\SS$, they carry a total order. In our notation, just as in \cite[Appendix~A]{BY}, $I_{d}$ is the lattice generated by the active exponents (for $G=\GL_n$).

    For a singular direction $d$, the partial order on the active exponents induced by $\prec_d$ on $\L_{d}$ is equivalent to giving the Stokes arrows in \cite[\S 5.5]{BoalchTop}.
\end{remark}

\subsection{Stokes conditions}
Given a vector space with a grading by some finitely generated abelian group, an ordering of this group determines a filtration of the vector space. In our more general situation we have the following result. 

Let $S \subset G$ be a torus. It acts on $\frg$ via the adjoint representation, and we denote by $\frg_{\r}$ the weight space for $\r \in X^*(S)$. Let $\L_S \subset X^*(S)$ be the subset of weights $\r$ for which $\frg_{\r}\neq 0$.

\begin{defn}
A \textit{set of positive weights for} $S$ is a subset $\Lambda_{S}^{+} \subset \Lambda_{S}\setminus \{0\}$ such that, setting $\Lambda_S^-\defeq \Lambda_S\setminus (\Lambda_S^+\cup\{0\})$, it induces a disjoint decomposition $\Lambda_{S}\setminus \{0\}=\Lambda_S^+\sqcup \Lambda_S^-$ into positive and negative weights satisfying the conditions:
\begin{enumerate} 
    \item For any $\l \in \Lambda_{S}$, either $\l\in \Lambda_{S}^{+}$ or $-\l\in \Lambda_{S}^{+}$.
    \item If $\l,\mu \in \Lambda_{S}^{+}$ and $\l+\mu \in \Lambda_{S}$, then $\l+\mu \in \Lambda_{S}^{+}$.
\end{enumerate} 
\end{defn}

Any such choice of $\Lambda_{S}^{+}$ determines a parabolic subgroup with Lie algebra $\frp = \bigoplus_{\rho\in\Lambda_S^+\cup \{0\}} \frg_{\r}$ of $\frg$ (and similarly for $\Lambda_{S}^{-}$). To see this, we may choose a maximal torus $T\subset G$ containing $S$. We let $\Phi$ be the set of roots of $G$ with respect to $T$, and let 
\[\Phi_{\ge 0} = \{\a \in \Phi \mid \a|_{S} \in\Lambda_S^+\cup \{0\} \}.\]
Similarly, $\Phi_{\le 0}$ is defined via $\Lambda_S^+$. Then clearly $\Phi_{\ge 0}$ and $\Phi_{\le 0}$ satisfy the following conditions:
\begin{enumerate}
    \item For every root $\alpha\in \Phi$, $\alpha \in \Phi_{\le 0}$ or $-\alpha \in \Phi_{\ge 0}$ (or both), and
    \item if $\alpha, \beta\in \Phi_{\ge 0}$ are such that $\alpha +\beta$ is a root, then $\alpha+\beta \in \Phi_{\ge 0}$.
\end{enumerate}
This means $\Phi_{\ge 0}$ is a parabolic subset of $\Phi$ in the sense of \cite[Definition 2.2.6]{CGP}, so $\frp$ is a parabolic subalgebra.

\begin{defn}
    Given a set of positive (resp.\ negative) weights $\Lambda_{S}^{+}$ (resp. $\Lambda_{S}^{-})$, we denote the corresponding parabolic subgroup by $P_{G}(\Lambda_{S}^{+})$ (resp.\ $P_{G}(\Lambda_{S}^{-})$).
\end{defn}

Let $S$ be a torus, and $\{S \hookrightarrow G\}/\Ad(G)$ a conjugacy class of embeddings of $S$ into $G$. Then the set of weights of the adjoint representation of $S$ on $\frg$ is well-defined. Let $\Lambda_{S}^{+} \subset \Lambda_{S}$ be a set of positive weights and choose an embedding $\g \colon S \hookrightarrow G$. Then by the above discussion, $\Lambda_{S}^{+}$ determines a parabolic subgroup $P_{G}(\Lambda_{S}^{+},\g)$ depending on the embedding $\g$. 

Assume we are given two sets of positive weights $\Lambda_{S,1}^{+}$ and $\Lambda_{S,2}^{+}$. 

\begin{defn}[Stokes condition for filtrations] Let $P,Q$ be two parabolic subgroups in $G$ in the same associate class.  We say that $P$ and $Q$ satisfy the \textit{Stokes condition} (for $\Lambda_{S,1}^{+}$ and $\Lambda_{S,2}^{+}$) if there is an embedding $\g \colon S \hookrightarrow G$ such that $P=P_{G}(\Lambda_{S,1}^{+},\g)$ and $Q=P_{G}(\Lambda_{S,2}^{+},\g)$.
\end{defn}
Again, when $G=\GL_n$, this coincides with the Stokes conditions in \cite[Definition 3.9]{BoalchTop}. 

\begin{lemma}\label{lem:relpos}
In the above situation, let $\WW_{P}$ and $\WW_{Q}$ be the parabolic subgroups in $\WW$ corresponding to the unique standard parabolic conjugate to $P$ and $Q$, respectively. There is a well-defined Weyl group element $w \in \WW_{P} \backslash \WW / \WW_{Q}$, determined uniquely by the pair $(\Lambda_{S,1}^{+},\Lambda_{S,2}^{+})$ and the conjugacy class of $S \hookrightarrow G$, such that the Stokes condition for $P,Q$ is equivalent to $P$ and $Q$ having relative position $w$ (meaning they lie in the orbit in $G/P \times G/Q$ labelled by $w$).
\end{lemma}
\begin{proof} Any representative $S \hookrightarrow G$ determines two parabolic subgroups corresponding to $\Lambda_{S,1}^{+}$ and $\Lambda_{S,2}^{+}$. We let $w \in \WW_{P} \backslash \WW / \WW_{Q}$ be their relative position. Choosing another representative amounts to simultaneous conjugation of the two parabolics. Therefore, $w$ is well-defined. It is clear that the statement holds for this $w$. 
\end{proof}

We wish to define a suitable Stokes condition for gradings as well. Abstractly, we are in the following situation. 

We call two embeddings $\g_1, \g_2 \colon S \hookrightarrow G$ \textit{compatible} if their images are conjugate, and if there is a choice of positive weights $\L^{+}_{S}$ such that $P_{G}(\L^{+}_{S},\g_{1}) = P_{G}(\L^{+}_{S},\g_{2})$. 

\begin{remark}\label{rem:wildmonodromy}
    The discussion in \cite[\S 3.4]{BoalchTop} readily generalizes to our setting. Let $P$ be a parabolic subgroup of $G$ with unipotent radical $U$. Denote by $\textup{Splits}(P)$ the set $\{\g\colon S \to G \mid P_G(\L^{+}_{S},\g)=P \}$. Then $\textup{Splits}(P)$ is a principal homogeneous $U$-space (for the conjugation action). Thus, there is a preferred unipotent element $g=g(\g_1, \g_2)$ in $G$ which conjugates $\g_1(S)$ into $\g_2(S)$. (In fact, this element only depends on $\gamma_1$ and $\gamma_2$, not on the choice of $\L^{+}_{S}$.) Following \textit{loc.~cit.}, we call this element the \emph{wild monodromy}.
\end{remark}

Assume now that in addition, we are given a partial order $\prec$ on $\L_S$, and that $\g_1, \g_2\colon S \to G$ are compatible. Moreover, let 
\[\Lambda_{S}^{\prec,+} = \{\lambda\in \L_{S} \mid 0 \prec \lambda \}.\]

\begin{defn}[Stokes condition for gradings]\label{defn:stokesgrcond}

    We say that two compatible maps $\g_1,\g_2$ satisfy the \textit{Stokes condition} if we have $P_{G}(\L^{+}_{S},\g_{1}) = P_{G}(\L^{+}_{S},\g_{2})$ for any set of positive weights $\L^{+}_{S}$ containing $\Lambda_{S}^{\prec,+}$.
\end{defn}

\begin{lemma}
If the preorder $\prec$ is empty, the Stokes conditions imply $\g_1(S)=\g_2(S)$.
\end{lemma}
\begin{proof}
Let $\L^{+}_{S}$ be any set of positive weights for which $P_{G}(\L^{+}_{S},\g_{1}) = P_{G}(\L^{+}_{S},\g_{2})$ (it exists by compatibility). Then, since $\prec$ is empty, we also have $\Lambda_{S}^{\prec,+}\subset \L^{-}_{S}$. Therefore, the wild monodromy $g$ lies in the unipotent radical of two opposite parabolics, hence must be trivial. 
\end{proof}

\subsection{Stokes filtrations and Stokes gradings}
In this section we define a generalization of Stokes-filtered and Stokes-graded local systems as defined in \cite[\S\S 6--7]{BoalchTop}. We will assume for simplicity of notation that $\ba$ is a singleton, so $\wdS$ has a single boundary circle $\partial$. Everything we do generalizes directly to an arbitrary number of punctures. 
\begin{remark} Let $X$ be a manifold, and let $Y\to X$ be a $G$-local system. Recall that a reduction of structure group to $H\subset G$ (or $H$-reduction) is an $H$-local system $E$ together with an isomorphism $E\times^{H}G \to Y$ of $G$-local systems. If $H$ is a closed subgroup, the set of $H$-reductions is in bijection with the global sections of $Y/H \to X$. In particular, if $Y\to X$ can be trivialized, and if we fix a trivialization, then an $H$-reduction is equivalent to a map $X \to G/H$. 
\end{remark}

Let $\Theta$ be an irregular class, with Stokes directions $\SS$. Outside the Stokes directions, we have a total ordering on $\L_{\Theta,d}$. The conjugacy class $\gamma \colon \cT_d \to G$ corresponding to $\Theta$ determines a conjugacy class of parabolic subgroups depending only on the component of $\partial \setminus \SS$ in which $d$ lies. This conjugacy class contains a unique standard parabolic subgroup. Enumerating $\partial \setminus \SS =  \partial_1 \cup \ldots \cup \partial_n$, we denote by $P_i$ the standard parabolic subgroup corresponding to $\partial_i$. 

\begin{defn} \label{defn:stokesfil} Let $\cF$ be a $G$-local system on $\wdS$, and let $\Theta $ be an irregular class. A \textit{Stokes filtration} of type $\Theta$ on $\cF$ is given by the following data: For each component $\partial_{i}$ of $\partial \setminus \SS$ a $P_i$-local system $E_i$ on $\partial_{i}$ and a reduction of structure group $\phi_{i}\colon E_{i} \times^{P_i} G  \cong \cF|_{\partial_{i}}$ such that the Stokes condition for filtrations is satisfied across each $d\in\SS$.

A \textit{Stokes-filtered $G$-local system} is a triple $(\cF, \Theta, F)$ consisting of a $G$-local system $\cF$ on $\wdS$, an irregular class $\Theta$ and a Stokes filtration $F$ on $\cF$ of type $\Theta$. 
\end{defn}
We spell out the meaning of Stokes conditions in this setting. For each $d\in \SS$ denote by $d_{1}$ and $d_{2}$ two directions in the neighboring components, ordered counter-clockwise. There are two corresponding total orders $\le_{1}$ on $\L_{\Theta,d_1}$ and $\le_{2}$ on $\L_{\Theta,d_2}$. The Stokes condition says that for $i=1,2$ there is a morphism $\cT_{d_{i}} \to \Aut(\cF_{d_{i}}) = \Aut(\cF_d)$ such that if for any trivialization $\cF_{d}\cong G$, the $P_{d_i}\subset G$ are the parabolic subgroups corresponding to the reduction of structure group on $\partial_{i}$, then $P_{d_i} = P(\le_{i})$.
\begin{remark} \label{rem:framebundle}
    Let us explain how to pass from a Stokes-filtered local system of rank $n$ as defined in \cite{BoalchTop} to a Stokes-filtered $\GL_n$-local system in the sense of the previous definition.

    Let $(V,\Theta,F)$ be a Stokes-filtered local system in the sense of \cite[\S 6]{BoalchTop}, i.e.\ $V\to \wdS$ is a local system of $\CC$-vector spaces, and $F$ is a collection of filtrations of $V|_{\partial_{i}}$, where $\partial \setminus \SS = \partial_{1} \cup \ldots \cup \partial_{r}$, satisfying the Stokes conditions.

    Consider the local system of sets $\cF = \Isom(V,\CC^{n}_{\wdS})$ defined by 
    \[U \mapsto \Isom(V|_{U}, \CC^{n}_{U}) \]
    for any open subset $U\subset \wdS$. Then this is a $\GL_n$-local system, as it carries a free fiberwise $\GL_n$-action by multiplying an isomorphism from the left. For each $i$, consider the local system $E_i\defeq\Isom^{F}(V|_{\partial_{i}}, \CC^{n}_{\partial{i}})$ with sections 
    \[\{\phi \in \Isom(V|_{U},\CC^{n}_{U} ) \mid \phi(F_jV|_{U})=\langle e_1,\ldots,e_{d_j} \rangle  \textup{ for all } j\} \]
    over any open $U\subset \partial_i$, where $d_j = \rk(F_j V|_{\partial_{i}})$. Let $P_i\subset \GL_n$ be the parabolic subgroup stabilizing the flag $W_j = \langle e_1,\ldots,e_{d_j} \rangle$ in $\CC^n$. Then $P_i$ acts on $E_i$, making it a $P_i$-local system on $\partial_i$. The natural map 
    \begin{align*}
        \phi_i\colon E_i\times^{P_i} \GL_n &\to \Isom(V, \CC^{n}_{\wdS} ) \\
        (\psi,g) &\mapsto g\psi 
    \end{align*}
    is an isomorphism, hence we obtain the required reduction of structure group $\phi_{i} \colon E_i\times^{P_{i}} \GL_n \cong \cF$. This determines a Stokes filtration $F'$ in the sense of Definition \ref{defn:stokesfil}, and $(\cF, \Theta, F')$ is a Stokes-filtered $\GL_n$-local system. Going back to a Stokes-filtered local system is done via the associated local system construction. 
\end{remark}

Let us generalize the notion of Stokes-graded local system to arbitrary $G$. 
\begin{defn}\label{defn:stokesgr} Let $\cF$ be a $G$-local system on $\wdS$, and let $\Theta $ be an irregular class, determining an algebraic quotient $T$ of $\cT$. A \textit{Stokes grading} $\Gamma$ of type $\Theta$ on $\cF$ is given by the following data: For each component $\partial_i'$ of $\partial \setminus \AA$ an embedding $T|_{\partial_i'} \hookrightarrow \Aut(\cF|_{\partial_i'})$ such that the Stokes condition for gradings holds across each $d\in \AA$.

A Stokes-graded $G$-local system on $\wdS$ is a triple $(\cF, \Theta, \Gamma)$ consisting of a $G$-local system $\cF$, an irregular class $\Theta$, and a Stokes grading $\Gamma$ of type $\Theta$.
\end{defn}
More precisely, we transport the gradings close to $d$ to the stalk at $d$, to make sense of the Stokes conditions as in Definition \ref{defn:stokesgrcond}, for the partial order $\prec_{d}$ which is refined by exponential dominance. 

\subsection{Comparison}

In this section we will lift the equivalence of connections, Stokes-filtered and Stokes-graded local systems on $\Sigma\setminus \ba$ in \cite{BoalchTop} to the $G$-versions. We proceed by using the Tannakian formalism (see \cite{DM}). 

Denote by $\Conn(\Sigma\setminus \ba)$ the category of algebraic connections on $\Sigma\setminus \ba$. The following lemma is well-known. For completeness, we explicitly state it here. A proof can be found in \cite[Théorème~4.2]{Mal2} (see also \cite[Ch.\ IV, (2.5)]{Mal}).  
\begin{lemma} \label{lem:malgrange}
    The functor associating to an algebraic connection $M\in \Conn(\Sigma\setminus \ba)$ its local system of flat sections equipped with the associated Stokes filtration is an equivalence of neutral Tannakian categories.
\end{lemma}
Additionally, denote by $\SFL(\Sigma, \ba)$ and $\SGL(\Sigma, \ba)$ the categories of Stokes-filtered and Stokes-graded local systems on $\wdS$. The following statement follows easily from \cite[Theorem 11.3]{BoalchTop}.
\begin{theorem}\label{thm:topstokes}
There is an equivalence of neutral Tannakian categories 
\begin{align}
\SFL(\Sigma, \ba) \cong \SGL(\Sigma, \ba).
\end{align} 
\end{theorem}

Moreover, it is well-known that the groupoid of $G$-connections $\Conn_{G}(\Sigma\setminus \ba)$ is naturally identified, via the associated bundle construction, with the groupoid of faithful tensor functors $\Fun^{\otimes}(\Rep(G),\Conn(\Sigma\setminus \ba))$.

Further denote by $\SFL_{G}(\Sigma, \ba)$ and $\SGL_{G}(\Sigma, \ba)$ the categories of Stokes-filtered and Stokes-graded $G$-local systems on $\widehat{\Sigma}$, respectively. We will prove the following. 
\begin{theorem}\label{thm:Comp}
    There are natural equivalences of categories 
    \begin{align}
        \Conn_{G}(\Sigma\setminus \ba) \cong \SFL_{G}(\Sigma, \ba)\cong \SGL_{G}(\Sigma, \ba).
    \end{align}
\end{theorem}
By Lemma \ref{lem:malgrange} and Theorem \ref{thm:topstokes}, it will suffice to prove the next Lemma. 
\begin{lemma}
    There are natural equivalences of groupoids 
    \begin{align}
        \SFL_{G}(\Sigma, \ba)&\cong \Fun^{\otimes}(\Rep(G),\SFL(\Sigma, \ba)), \textup{and} \\
        \SGL_{G}(\Sigma, \ba)&\cong \Fun^{\otimes}(\Rep(G),\SGL(\Sigma, \ba)).
    \end{align}
    
\end{lemma}
\begin{proof}
We only prove the first equivalence. The proof of the second one is analogous. We start with the construction of a tensor functor out of a Stokes-filtered $G$-local system. So let $(\cF, \Theta, F)$ be a Stokes-filtered $G$-local system. Recall that $\cF$ is a $G$-local system on $\wdS$, $\Theta$ is an irregular class, and $F$ is a Stokes filtration. For $\partial \setminus \SS = \partial_{1}\cup \ldots \cup \partial_{r}$, it consists of a $P_i$-torsor $E_i$ and a reduction of structure group $\phi_{i} \colon E_{i}\times^{P_{i}} G \cong \cF|_{\partial_{i}}$ on each $\partial_{i}$. 

Now let $(\pi, V)$ be a complex representation of $G$. Then $\cF \times^{G} V$ is a local system of $\CC$-vector spaces with fiber $V$ on $\wdS$. We need to construct an irregular class (of total dimension $\dim(V)$) for it, and a Stokes filtration in the sense of \cite[\S 6]{BoalchTop}. For the irregular class, choose any $d\in \partial$ and any representative $\cT_d \to G$. Composing with $\pi$ determines a map $\cT_d \to \GL(V)$, and we let $\Theta_{V}$ be its conjugacy class. 

We now explain how to construct the Stokes filtration on $\cF \times^{G} V$. 
The irregular class $\Theta$ determines a standard parabolic $P_i \subset G$ for any $i=1,...,r$, and similarly $\Theta_{V}$ determines a parabolic subgroup $P_{V,i} \subset \GL(V)$ for all $i$.  

For simplicity, let us drop the index $i$ from the notation whenever convenient for now. It is clear that $\rho(P)\subset P_{V}$. The reduction of structure group $\phi\colon E \times^{P} G \cong \cF|_{\partial_{i}}$ induces an isomorphism
\[\cF\times^{G} V |_{\partial_{i}}\cong (E \times^{P} G) \times^{G} V \cong E\times^{P} V, \]
and further 
\[E\times^{P} V \cong E \times^{P} P_{V} \times^{P_{V}} V,\]
where $E_{V} \defeq E \times^{P} P_{V}$ is now a $P_{V}$-local system. Write $P_{V} = \Stab({F_j})$ with $F_j\subset V$. Then we obtain a filtration of $\cV \defeq \cF\times^{G} V |_{\partial_{i}}$ by defining $F_j\cV \defeq E_{V} \times^{P_{V}} F_j$. 

We leave it to the reader to verify that this defines a Stokes filtration on $\cF \times^{G} V$, and that we have obtained the required faithful tensor functor.

We now reconstruct a Stokes-filtered $G$-local system from a given faithful tensor functor $\Phi\colon\Rep(G) \to \SFL(\Sigma, \ba)$. In this situation, for any complex representation $(\pi,V)$ we are given the following data. A local system of vector spaces $\cF_{V}$ with fiber $V$ on $\wdS$, an irregular class $\Theta_{V}$ of total dimension $\dim(V)$, and a Stokes filtration of type $\Theta_{V}$ on $\cF_{V}$.

Composing $\Phi$ with the forgetful functor $\SFL(\Sigma, \ba) \to \Loc(\wdS)$ we obtain a faithful tensor functor $\Rep(G) \to \Loc(\wdS)$, which is the same as a $G$-local system $\cF$ on $\wdS$ (see e.g.\ \cite[§6]{Sim1}).

To reconstruct an irregular class, we use the functor 
\[\gr\colon \SFL(\Sigma, \ba) \to \{\cI\textup{-graded local systems on $\partial$}\},\]
associating to each Stokes-filtered local system the associated $\cI$-graded local system on $\partial$, cf.\ \cite{Mal2}. Taking the stalk at an arbitrary $d \in \partial$ we obtain a faithful tensor functor from $\SFL(\Sigma, \ba)$ to the category of $\cI_d$-graded vector spaces. Its Tannaka group is $\cT_d$, so we obtain a map $\cT_d \to G$ by composing with the given functor $\Phi$, well-defined up to conjugation (because we chose the basepoint $d$). The conjugacy class of $\cT_d \to G$ determines an irregular class $\Theta$.

Finally, we will reconstruct the Stokes filtration, and we will use Theorem \ref{thm:topstokes} for this. We actually reconstruct a Stokes grading, and pass to the associated Stokes filtration. Recall that we denote by $\AA = \AA_{\Theta}$ the singular directions for the irregular class constructed above. Write $\partial \setminus \AA = \partial_1' \cup \ldots \cup \partial_r'$.

For each $i=1,\ldots,r$ choose a direction $d_i \in \partial_i'$. It is easy to check that $d_i$ is not a singular direction for $\Theta_{V}$. 
Thus, composing $\Phi$ with the equivalence $\SFL(\Sigma, \ba) \cong \SGL(\Sigma, \ba)$, we obtain a faithful tensor functor from $\Rep(G)$ to $\cI|_{\partial_i'}$-graded local systems on $\partial_i$ for every $i$. The discussions following \cite[Lemma 6.12]{Sim1} can be adapted to see that giving such a tensor functor is equivalent to giving an $\cI|_{\partial_i'}$-grading on $\cF|_{\partial_i'}$. Indeed, the grading gives for every direction $d \in \partial_i'$ a morphism $\cT_d \to \Aut^{\otimes}(\cF \times^G (-))$, and varying $d$ one obtains $\cT|_{\partial'_i} \to \Aut(\cF|_{\partial_i'})$. One checks that in this way, we have obtained a Stokes-graded $G$-local system $(\cF, \Theta, \Gamma)$, satisfying in particular the Stokes conditions for gradings. 

To conclude, we now pass to a Stokes filtration. For that, we use the following general lemma. 

\begin{lemma}\label{lem:GfiltrGgrad} 
    Let $\cF$ be an $\cI$-graded $G$-local system on a well-behaved topological space $X$ with local isomorphism class $\Theta$. Assume that $\L_{\Theta,x}$ has a set of positive weights $\Lambda_x^+$ for every $x\in X$, which varies locally constantly. This determines a standard parabolic subgroup $P\subset G$. Then, $\cF$ has a natural structure of $\cI$-filtered $G$-local system, i.e.\ there is a reduction of structure group $\phi\colon E \times^{P} G \xrightarrow{\sim} \cF$, where $E$ is a $P$-local system on $X$.
\end{lemma}
\begin{proof}
    Denote by $\cF_{\triv}$ the trivial $G$-local system on $X$, and choose a trivialization $\cF \cong \cF_{\triv}$. Then this identifies $\Aut(\cF) \cong G$. We are given a morphism $\cT \to \Aut(\cF)\cong G$, determining a parabolic subgroup $P_{\Aut(\cF)}(\Lambda^+)$ in $\Aut(\cF)$, and hence in $G$. (Here, $\Lambda^+$ is defined by the locally constant $\Lambda_x^+$.)
    Let $E=\Isom^{P}(\cF_{\triv},\cF)$ be the set of trivializations of $\cF$ which map $P_{\Aut(\cF)}(\Lambda^+)$ to $P$. It is easy to check that this is a $P$-local system, and that the natural map 
    \[E \times^{P} G \to \cF\]
    is an isomorphism. This way we obtain the required reduction of structure group.
\end{proof}
Applying Lemma~\ref{lem:GfiltrGgrad} on each component of $\partial \setminus (\AA \cup \SS)$, we obtain on each of them a standard parabolic subgroup, and a reduction of structure group. The Stokes conditions for gradings guarantee that these data agree when crossing a singular direction in $\AA$. We have thus constructed the required reduction of structure group on each component of $\partial \setminus \SS$. These data satisfy the Stokes conditions for filtrations by construction. We leave it to the reader to check this construction is quasi-inverse to the construction of the tensor functor in the first part of the proof. 
\end{proof}

\subsection{Stokes $G$-local systems}
Let us also make the link with the notion of Stokes $G$-local system from \cite{BY}. Recall from \emph{loc.~cit.}\ that, given a Riemann surface $\Sigma$ with a marked point $\ba=\{a\}$ and an irregular class $\Theta$, one denotes by $\widehat{\Sigma}$ its real oriented blow-up at $\ba$ with boundary circle $\partial$. One then defines $\widetilde{\Sigma}$ to be the surface which is obtained as follows: Consider a small open neighbourhood $\HH$ of $\partial$ in $\widehat\Sigma$, bounded by $\partial$ and another circle $\widetilde{\partial}$. We can choose a homeomorphism $e\colon \partial \To{\sim} \widetilde\partial$, respecting the orientation (for example, choose $e(d)$ and $d$ to have the same argument in some local coordinate around $a$). Then, remove from $\widehat{\Sigma}$ the point $e(d)$ on $\widetilde{\partial}$ for each singular direction $d\in\AA\subset \partial$.

If $d\in \AA$, we will denote by $\gamma_d$ a small loop based at $d\in \partial$ going around the point $e(d)$ in a positive sense (and around no other of the punctures introduced above).

For a local system $\cL$ on $\widetilde\Sigma$ and any $d\in \AA$, we define the Stokes group $\SS \mathrm{to}_d\subset \Aut(\cL_d)$ as follows: Let $\ell$ be the Lie algebra of $\Aut(\cL|_d)$. We have a decomposition $\ell=\bigoplus_{\lambda\in I_d} \ell_\lambda$ and we define $\SS\mathrm{to}_d$ to be the product of the groups $\exp(\ell_\lambda)$ for all $\lambda$ with maximal decay at $d$.

\begin{defn}[{\cite[Definition~13]{BY}}]\label{def:StokesLocSys}
    A Stokes $G$-local system with irregular class $\Theta$ is a $G$-local system $\cS$ on $\widetilde{\Sigma}$ such that the following properties hold:
    \begin{itemize}
        \item[(1)] the local system $\cS|_\partial$ is $\cI$-graded with irregular class $\Theta$, and
        \item[(2)] for any $d\in \AA$, the monodromy of $\cS_d$ around $\gamma_d$ is an element of $\SS \mathrm{to}_d$.
    \end{itemize} 
\end{defn}
Let us denote by $\mathrm{SLoc}_G(\Sigma,\mathbf{a})$ the category of Stokes $G$-local systems as in the definition above (with arbitrary irregular class). Note that the following result was also obtained independently in \cite[Corollary 2.15.]{HS}. 

\begin{prop}\label{prop:EquivStokesGLS}
    The category of Stokes-graded $G$-local systems is equivalent to the category of Stokes $G$-local systems.
    As a consequence, there is an equivalence of categories $\Conn_G(\Sigma\setminus\ba)\cong \mathrm{SLoc}_G(\Sigma,\ba)$.
\end{prop}
\begin{proof}
    Given a Stokes $G$-local system $\cS$, we want to define a Stokes-graded $G$-local system $\cF$ from it.		
    Recall that $\cS$ is a local system on $\widetilde{\Sigma}$. If we restrict it to $\widetilde{\Sigma}\setminus \HH\subset \widehat{\Sigma}$, we get a local system that has an obvious natural extension to $\widehat{\Sigma}$ (since the latter is homotopy equivalent to $\widetilde{\Sigma}\setminus \HH$), and we define $\cF$ to be this local system.
    
    From the datum of the given Stokes $G$-local system, we have in particular a morphism of local systems of groups
    $$\cT\to\Aut(\cS|_\partial)$$
    factoring through some algebraic quotient $T$ of $\cT$.
    Since for any $d\notin \AA$, there is a natural isomorphism $\cS_d\simeq \cF_d$ (given by following a straight path from $d$ through $e(d)$), we get for any $d\notin\AA$ a morphism $T_d\hookrightarrow \Aut(\cF_d)$.
    
    It remains to prove that these gradings satisfy the Stokes condition at each $d\in\AA$: Let $d\in \AA$ be a singular direction, then we obtain two morphisms $T_d\hookrightarrow \Aut(\cF_d)$ from the ones on both sides of $d$. We can identify $\Aut(\cF_d)\simeq \Aut(\cS_d)\simeq G$, and hence one obtains two tori $S_1,S_2\subset G$. (Note that this identification involves a choice, which does, however, not affect what follows.) From the definition of a Stokes $G$-local system, we see that these tori are related by $S_2=sS_1s^{-1}$ for some $s\in\SS\mathrm{to}_d$. One can then easily identify the character groups $X^*(S_1)$ and $X^*(S_2)$. Moreover, it is not difficult to see that the weight spaces with respect to $S_1$ and those with respect to $S_2$ (for the adjoint action of $G$ on $\frg$) are related via $\Ad_{s}$, and that one can therefore identify the lattices $\Lambda_{S_1}$ and $\Lambda_{S_2}$. Let $\Lambda^+$ be a set of positive weights on $\Lambda\defeq \Lambda_{S_1}\cong \Lambda_{S_2}$ compatible with the partial order $\prec_d$ given by maximal decay (cf.\ \S \ref{sec:stokesdir}). The compatibility with $\prec_d$ means that $\Lambda^{\prec_d,-}\defeq \{\lambda\in\Lambda\mid \lambda\prec_d 0\}\subset \Lambda^-$. Then (using the notations from \S 3.4), we get from the above that $\frp_2 = \Ad_{s}(\frp_1)$. Now, by definition, we have $\mathbb{S}\mathrm{to}_d=\exp(\bigoplus_{\rho\in \Lambda, \rho\prec_d 0} \frg_\rho)\subset P_G(\Lambda^-)$ because of the compatibility of $\Lambda^-$ with $\prec_d$. Hence, $\Ad_s$ acts on $\frp_1$, i.e.\ $\frp_2=\Ad_s(\frp_1)=\frp_1$, as desired.

    Conversely, assume that we are given a Stokes-graded $G$-local system $\cF$. This means that on any connected component $\partial_i\subset \partial\setminus \AA$, we have a morphism $\cT|_{\partial_i}\to\Aut(\cF|_{\partial_i})$, which extends naturally to a morphism $\cT|_{\overline{\partial_i}}\to \Aut(\cF|_{\overline{\partial_i}})$ on the closure of $\partial_i$. In particular, at a singular direction $d\in\AA$, we obtain two compatible morphisms $T_d\to \Aut(\cF_d)$, and hence a unipotent element $g_d\in\Aut(\cF_d)$ (cf.\ Remark~\ref{rem:wildmonodromy}). We use this element to glue $\cF|_{\overline{\partial_i}}$ and $\cF|_{\overline{\partial_{i+1}}}$ to a local system on $\overline{\partial_i}\cup\overline{\partial_{i+1}}$ (if $\partial_i$ and $\partial_{i+1}$ are adjacent to the singular direction $d$). In this way, we obtain a local system $\cS^0$ on $\partial$, and by construction it comes with a global grading $\cT\to\Aut(\cS^0)$.

    Now we construct the Stokes local system $\cS$: We can naturally extend $\cS^0$ as a local system on the halo $\HH$. On the other hand, the local system $\cF$ can be considered as a local system outside the halo. Both are glued to each other in the obvious way to obtain a local system on $\widetilde\Sigma$: The boundary of $\HH$ consists of segments bounded by punctures corresponding to consecutive singular directions. In each such segment, $\cS^0$ is by construction identified with the restriction of $\cF$ on this segment.

    It remains to show that $\cS$ satisfies the second condition in Definition~\ref{def:StokesLocSys}. This comes down to proving that for each $d\in\AA$, the above element $g_d$ lies in $\SS \mathrm{to}_d$. Let $\Lambda_{S_1}^+$ be a set of positive weights of $\Lambda_{S_1}$ compatible with $\prec_d$, and let $\Lambda_{S_2}^+$ be the corresponding set of positive weights of $\Lambda_{S_2}^+$ via the isomorphism $\Lambda_{S_1}\cong \Lambda_{S_2}$. Since $S_2=g_dS_1g_d^{-1}$, we can conclude (similarly as in the first part of the proof) that $P_G(\Lambda_{S_2}^+) = g_d P_G(\Lambda_{S_1}^+) g_d^{-1}$ . On the other hand, we have $P_G(\Lambda_{S_2}^+) = P_G(\Lambda_{S_1}^+)$ by the Stokes condition for gradings. Hence, $g_d$ is an element in the normalizer of $P_G(\Lambda_{S_1}^+)$, which coincides with $P_G(\Lambda_{S_1}^+)$ since it is a parabolic subgroup. Since this argument holds for any set of positive weights compatible with $\prec_d$, we get $g_d\in \bigcap\limits_{\Lambda_{S_1}^+ \text{ extends } \prec_d} P_G(\Lambda_{S_1}^+)$. It is easy to check that the unipotent radical of this intersection is precisely the Stokes group $\SS \mathrm{to}_d$ (cf.\ \cite[Lemma 3.8]{BoalchTop} for the $\GL_n$-case). Therefore, we have $g_d \in \SS \mathrm{to}_d$.

    One can check that the two constructions given here are inverse to each other, and this concludes the proof.
\end{proof}

Note that the equivalence $\Conn_G(\Sigma\setminus\ba)\cong \mathrm{SLoc}_G(\Sigma,\ba)$ has also been proved in \cite[Appendix~A]{BoalchGeomBraiding} for unramified irregular classes.

\section{Braids and Stokes filtrations}\label{s:braids}

\subsection{Braid varieties and stacks} In this section, we relate Stokes-filtered $G$-local systems to certain Betti moduli spaces attached to braids, considered in \cite{BBMY}. Let us first recall the definition of these spaces. For that, recall that we fixed a maximal torus $T$ and a Borel subgroup $B$ and that we can identify $W\cong \WW$ using this coice. We will only use the notation $W$ from now on for simplicity.

The pair $(W, S)$ denotes the Weyl group for $G$ with generating set of simple reflections corresponding to $B$. Let $\Br_{W}$ be the Artin braid group for $(W,S)$ with positive braid monoid $\Br^{+}_{W}$. The group $W$ is recovered as a quotient $\Br_{W} \to W$ and each $w\in W$ has a canonical lift $\tilde{w}$ in $\Br^{+}_{W}$, given as a reduced word in elements of $S$. 

In \cite{BBMY}, the authors associate to each braid $\beta \in \Br^{+}_{W}$ a smooth algebraic stack $\cM(\beta)$ over $\CC$, which is defined as follows: Write $\beta = \tilde{w}_{1}\cdot \ldots \cdot \tilde{w}_{n}$ for some sequence of elements $w_1,\ldots,w_n\in W$. Let $Y$ be any $\CC$-scheme, then $\cM(\beta)(Y)$ is the groupoid with the following objects (and obvious morphisms):
\begin{enumerate}
\item An $(n+1)$-tuple $(E_0,\ldots,E_n)$ of $B$-torsors over $Y$,
\item for all $0\le i \le n-1$, isomorphisms $\iota_{i}\colon E_i \times^{B} G\To{\sim} E_{i+1} \times^{B} G$ such that under this identification both $B$-reductions are in relative position $w_i$, and
\item an isomorphism $\tau\colon E_n \To{\sim} E_0$.
\end{enumerate}

\begin{remark} 
    Any object of $\cM(\beta)(Y)$ is isomorphic to one where $E_n=E_0$, and $\tau = \id$, by replacing $\iota_{n-1}$ with $\tau^{B} \circ \iota_{n-1}$. Here, $\tau^{B}$ denotes the isomorphism $E_n \times^{B} G \To{\sim} E_0\times^{B} G$ induced by $\tau$.
\end{remark}
Alternatively, $\cM(\beta)$ can be identified with a quotient stack as follows. We let 
\begin{align}\label{eqn:globalquotient}
    \cM^{\sharp}(\beta) = \{(B_0,\ldots,B_n, h) \in \cB^{n+1} \times G \mid (B_i, B_{i+1})\in \cO(w_i), B_n = {}^{h}B_0 \}.
    \end{align}
The group $G$ acts on $\cM^{\sharp}(\beta)$, namely it acts diagonally on the $B_i$, and by conjugation on $h$. The quotient stack $[G\backslash \cM^{\sharp}(\beta)]$ is identified with $\cM(\beta)$. 

\begin{remark}\label{rem:braidindep}
    The spaces $\cM(\b)$ and $\cM^{\sharp}(\b)$ depend only on the braid $\b$, not on the chosen decomposition. More precisely, when $w=w_1w_2$, then $\cM(\tilde{w})$ and $\cM(\tilde{w}_{1}\tilde{w}_{2})$ are canonically isomorphic. This follows from work of Deligne \cite{Deligne}. 
\end{remark}

\subsection{Moduli stack of Stokes-filtered $G$-local systems} Let $\Sigma = \PP^1$, $\ba = \{0,\infty\}$, fix the tame irregular class at $0$, and fix an irregular class $\Theta$ at $\infty$. Denote by $\partial$ the boundary circle at $\infty$. We can slightly adapt the above construction to define a moduli stack of Stokes-filtered $G$-local systems on $\wdS$. Let $\SS$ be the set of Stokes directions for $\Theta$, and fix a non-Stokes direction $d_0\in\partial$. This allows us to enumerate $\SS = \{d_1,\ldots,d_n \}$ in a counter-clockwise way, starting from $d_0$. We denote the open component between $d_i$ and $d_{i+1}$ by $\partial_i\subset \partial$.

Recall that in addition to a standard parabolic subgroup $P_i$ for each component $\partial_i$ of $\partial$, we have a well-defined parabolic subgroup $W_i \subset W$ corresponding to $P_i$, and the relative position $w_i \in W_i \backslash W / W_{i+1}$ of the exponential dominance orderings across the direction $d_i$. Note that no relative position $w_i$ can be trivial.

Now, a Stokes-filtered $G$-local system $(\cF, \Theta, F)$ is given by a $G$-local system $\cF$ on $\PP^1 \setminus \{0,\infty\}$, and a reduction of structure group $\phi_i \colon E_i \times^{P_i} G \cong \cF|_{\partial_i} $ on each component $\partial_i$ of the boundary circle $\partial$ at $\infty$. The boundary circle is a homotopy retract of $\wdS$, so we can think of $\cF$ as a $G$-local system on $\partial$. Since each component is contractible, the datum of $\phi_i$ is equivalent to giving a reduction of structure group of a $G$-torsor on a point. These data satisfy the Stokes condition, which is equivalent to having relative position $w_i$. To specify the local system on the circle, we only need to specify gluing maps across each Stokes direction. We thus make the following definition. 

Let $Y$ be any complex test scheme. Then we define $\cM_{B,\Theta}(Y)$ to be the groupoid whose objects are tuples $(E_1,\ldots,E_n)$ where each $E_i$ is a $P_i$-torsor on $Y$, together with isomorphisms $\iota_i \colon E_i \times^{P_i} G \cong E_{i+1}\times^{P_{i+1}} G$ for $i=1,\ldots,n-1$, and an isomorphism $\iota_n \colon E_n \times^{P_n} G \cong E_{1}\times^{P_{1}} G$, such that via these maps, the reductions of structure group are in relative position $w_i$. Recall that all $P_i$ have the same associate class. The following lemma easily follows from a description of $\cM_{B,\Theta}$ similar to \eqref{eqn:globalquotient}. Indeed, you can choose a trivialization of $E_1 \times^{P_1} G$, to obtain a trivialization for every $E_i \times^{P_i} G$, and hence a tuple of parabolic subgroups. Forgetting the trivialization amounts to quotienting by the diagonal action of $G$. 
\begin{lemma}
For any irregular class $\Theta$, $\cM_{B,\Theta}$ is an algebraic stack.
\end{lemma}
The stack of Stokes-filtered $G$-local systems $\cM_{B,\Theta}$ is equipped with a monodromy map
\[\cM_{B,\Theta} \to [G /\Ad(G)], \]
which encodes the monodromy around $0$ (sometimes called topological monodromy in this setting).

The direct relation to the space $\cM(\beta)$ can now be recovered as follows. Assume $\Theta$ is a toral irregular class. In this case, each relative position $w_i$ is an element in $W$. This way, we can attach a braid to $\Theta$ by defining
\[\beta_{\Theta} = \tilde{w}_{1}\cdot \ldots \cdot \tilde{w}_{n}.\]
Let $H = B/[B,B]$ be the universal Cartan of $G$. When $\Theta$ is toral, one may define a formal monodromy map $\cM_{B,\Theta} \to [H / \Ad_{w}(H) ]$, where $w$ is the image of $\beta_{\Theta}$ under the projection $\Br_{W} \to W$, and $\Ad_{w}$ denotes $w$-twisted conjugation. The topological monodromy and formal monodromy maps are constructed completely analogously to the definition in \cite[\S 4.1]{BBMY}.
\begin{prop}
    If $\Theta$ is toral, there is a natural isomorphism of stacks
    \[ \cM_{B,\Theta} \xlongrightarrow{\cong} \cM(\beta_{\Theta}) \]
    which is compatible with the monodromy maps.
\end{prop}
\begin{proof}
    Since $\Theta$ is toral, a Stokes-filtered $G$-local system of type $\Theta$ is precisely given by a tuple $(E_0,\ldots,E_{n-1})$ of $B$-torsors (on each sector of $\partial\setminus\SS)$ in relative position $w_i$. We define $E_n=E_0$, and $\tau = \id$. Moreover, we have isomorphisms $\iota_{i}\colon E_i \times^{B} G\to E_{i+1} \times^{B} G$ obtained from the reduction of structure group of the $G$-local system on each sector. It is clear that this is an isomorphism and is compatible with monodromy. 
\end{proof}

\subsection{Isoclinic irregular classes}

Recall that $T\subset G$ is a maximal torus, so that $\frt\subset \frg$ is a Cartan subalgebra, and let $D_{\infty}^{\times} = \Spec \CC(\!(t)\!)$. We recall the notion of isoclinic irregular type from \cite[\S 2.1]{JY}. Let $\CC[t^{\QQ_{<0}}]$ be the set of $\CC$-linear combinations of monomials $t^{a}$ with $a\in \QQ_{<0}$, and let $\frt[t^{\QQ_{<0}}] = \frt \otimes_\CC \CC[t^{\QQ_{<0}}]$. Then, an element $X\in \frt[t^{\QQ_{<0}}]$ is called an isoclinic irregular type of slope $\nu \in \QQ_{>0}$ if 
\begin{enumerate}
    \item the lowest degree term $X_{-\nu}\in \frt$ of $X$ is regular semisimple, and
    \item there exists $n\in \ZZ_{> 0}$ and a formal connection $(E,\nabla)$ over $D_{\infty}^{\times}$ such that $X\in \frt[t^{-1/n}]$ and the pullback of $(E,\nabla)$ to the $n$-fold cover of $D_{\infty}^{\times}$ is isomorphic to a connection of the form 
\[d+ (X+\frg[\![t^{1/n}]\!]) \frac{dt}{t}. \]
\end{enumerate}
A formal connection $(E,\nabla)$ on $D_{\infty}^{\times}$ is isoclinic of slope $\nu$ if there is an integer $n$ such that $(E,\nabla)$ takes the above form on the $n$-fold covering. The denominator of $\nu$ (written in lowest terms) is necessarily a regular number for $G$. An isoclinic irregular class is the irregular class of an isoclinic connection. Explicitly, they are given as follows.

Let $\gt = \frg\otimes_{\CC} \CC(\!(t)\!)$, and let $\dt\subset \gt$ be a Cartan subalgebra, possibly non-split. Let $w$ be a root of $t$ such that $\dt$ splits in $\gt \otimes_{\CC} \CC(\!(w)\!)$, i.e.\ $\dt \cong \frt(\!(w)\!)$. The $w$-adic filtration on $\frt(\!(w)\!)$ induces a canonical filtration on $\dt$, and we denote by $\dt_{<0}$ its associated graded in negative degrees. Let $m$ be a regular number for $G$, and let $d\in \ZZ_{<0}$ be coprime to $m$. Recall that an element in $\dt/\dt_{\ge 0}$ determines an irregular class, and vice versa every irregular class is of this form for some Cartan, cf.\ \cite[\S 3.5]{BY}.

\begin{defn}
We say that $\Theta$ is an \textit{isoclinic irregular class of slope $\nu=\frac{d}{m}$} (with $d,m$ coprime) if there is a Cartan $\dt\subset \gt$, split by adjoining $w=t^{1/m}$, and an $X\in \dt_{<0}$ which is conjugate to an isoclinic irregular type
\[X_{-d/m}t^{-d/m}+X_{-(d-1)/m}t^{-(d-1)/m}+\ldots+X_{-1/m}t^{-1/m}\in \frt(\!(w)\!)\] 
with $X_{-d/m}$ being regular semisimple, such that $\Theta$ corresponds to the class of $X$ in $\dt/\dt_{\ge 0}$. Note that every isoclinic class is automatically toral. We call an isoclinic class \textit{elliptic} if moreover $m$ is a regular elliptic number for $G$. 
\end{defn}

\begin{remark}
    Let $D_{\infty}^{\times} = \Spec \CC(\!(t)\!)$, and assume that $G$ is semisimple and $m$ is elliptic. In the proof of \cite[Lemma 2.10]{JY} it is shown that in this case there is up to isomorphism a unique isoclinic $G$-connection of slope $\nu = d/m$ on $D_{\infty}^{\times}$ with fixed isoclinic irregular type. 
\end{remark}
The following is explained in \cite[\S 4.4]{BBMY}. 
\begin{lemma}\label{lem:isoclinicbraid}
Let $\Theta=\Theta_{\nu}$ be an isoclinic irregular type of slope $\nu = d/m$, written in lowest terms. Assume $m$ is regular elliptic. Then, up to a cyclic shift of words, $\beta_{\nu}\defeq\beta_{\Theta} = \tilde{w}^d$, for a regular elliptic element $w$ of minimal length in its conjugacy class. 
\end{lemma}
Note that two braids which differ only by a cyclic shift of words give rise to isomorphic $\cM(\beta)$.

\section{Physical rigidity of Airy and Kloosterman connections}\label{s:rigidity}

\subsection{Kloosterman connections} Assume for the rest of the paper that $G$ is a simple complex group with Coxeter number $h$. Let $\Sigma = \PP^1$, and let $\ba = \{0,\infty\}$. We will be interested in algebraic connections on $\Sigma \setminus \ba$ which are regular singular at $0$, and have slope $1/h$ at $\infty$. 

\begin{defn}
    A Kloosterman $G$-connection (or generalized Frenkel--Gross connection) is an algebraic $G$-connection on $\PP^1 \setminus \{0,\infty\}$ which is isoclinic of slope $1/h$ at $\infty$, and regular singular at $0$. 
\end{defn}
These connections were first constructed as de Rham analogues of $\ell$-adic Kloosterman sheaves in \cite{HNY}. They generalize the Frenkel--Gross connection in \cite{FG}, and can be written down explicitly, see for example \cite{KS}. The local monodromy of a Kloosterman $G$-connection at $0$ is automatically regular. Before we study them further, we explain some generalities which allow us to describe their moduli explicitly. 

\subsection{Generalized Steinberg sections}

Let $w\in W$ be an elliptic element of minimal length in its conjugacy class, and denote by $T^{w}$ the fixed points in $T$ for the natural adjoint action of $w$ on $T$. We fix a lift of $w$ in $N(T)$ which we denote by the same symbol. 

Let $\Sigma_{w} = (U\cap wU^{-}w^{-1}) w = Uw \cap wU^{-}$. Note that there is an action of $T^{w}$ on $\Sigma_{w}$ by conjugation (because $T$ normalizes $U$ and $U^{-}$, and every $t \in T^{w}$ is fixed by $w$).

The following statements are the main results of He and Lusztig's paper \cite{HL}:
\begin{enumerate}
    \item The conjugation action of $U$ on $UwU$ is free. 
    \item The subset $\Sigma_{w}$ meets every $\Ad(U)$-orbit on $UwU$ in a single point. 
    \item To summarize, the map $\Sigma_{w} \to UwU / \Ad(U)$ identifies the latter quotient with an affine space of dimension $\ell(w)$.
\end{enumerate}
Here recall that $\ell(w)$ is the number of positive roots made negative by $w$ or equivalently the number of positive roots made negative by $w^{-1}$, and the latter are precisely the roots occuring in $U \cap wU^{-}w^{-1}$. 

When $w$ is the lift of a Coxeter element, $\Sigma = \Sigma_{w}$ is the usual Steinberg cross-section. The name cross-section is explained by the following result. 

\begin{theorem}[\cite{St}]\label{thm:steinberg} Let $w$ be a Coxeter element. When $G$ is simply connected, $\Sigma$ is a cross-section for the regular conjugacy classes in $G$. That is, $\Sigma$ intersects every regular conjugacy class in a unique point. 
\end{theorem}

Moreover, to handle the general case when $G$ is not necessarily simply connected, we prove the following. 

\begin{theorem} \label{thm:Twtrans}
Let $w$ be a Coxeter element. For any regular class $\cC$, the group $T^{w}$ acts transitively on $\cC\cap \Sigma$.
\end{theorem}

\begin{proof}
Let $\pi\colon G^{\sc} \to G$ be the simply connected (universal) cover of $G$. The map $\pi$ is a central isogeny, and we denote its kernel by $K \subset Z(G^{\sc})$. Note that $K \cong \pi_{1}(G)$.

Let $\Sigma^{\sc}$ be the Steinberg section in $G^{\sc}$ defined using a lift of $w$ in $G$ to $G^{\sc}$. Then $\pi(\Sigma^{\sc}) = \Sigma$. Let $\cC$ be a regular conjugacy class in $G$. 

Note that the restriction of $\pi$ to $\Sigma^{\sc}$ is injective. Indeed, if $s,s'\in \Sigma^{\sc}$ with $\pi(s) = \pi(s')$, it means $s' = ks$ for some $k \in K$. Since $\Sigma^{\sc}= Uw \cap wU^{-}$, we can write $s = uw$ and $s' = vw$. This implies $vw = kuw$, hence $k = vu^{-1}\in U$. However, since $k$ is central, it is in particular semisimple, so $k=1$ and $s'= s$.

We describe the intersection $\cC \cap \Sigma$. Fix $y_1 \in \cC \cap \Sigma$, then we can find a unique $x_1 \in \Sigma^{\sc}$ with $\pi(x_1) = y_1$. 

Note that $x_1$ is regular in $G^{\sc}$, and so is $kx_1$ for every $k\in K$. We can thus define $x_k \in \Sigma^{\sc}$ to be the unique element in $\Sigma^{\sc}$ conjugate to $kx_1$ (cf.\ Theorem~\ref{thm:steinberg}). The $x_k$ need not be pairwise distinct. In fact, the map 
\[\{x_k \mid k \in K \} \to \{G^{\sc}.x_{k} \mid k \in K \} \] 
is a bijection onto the set of conjugacy classes in $G^{\sc}$ lying over $\cC$: It is easily checked that the latter are exactly the conjugacy classes of the elements $kx_1$ for $k\in K$. Moreover, by definition, the points $x_k$ and $x_{k'}$ coincide exactly when $kx_1$ is conjugate to $k'x_1$. The set of conjugacy classes in $G^{\sc}$ lying over $\cC$ is a $\pi_1(G)/\pi_0(C_G(y_1))$-torsor by \cite[Lemma 2.2.5]{JY2}. Under the above identification $\pi_1(G)\cong K$, the subgroup $\pi_0(C_G(y_1))\subset \pi_1(G)$ corresponds to the subgroup 
\[K(x_1) = \{k \in K \mid kx_1 \textup{ is conjugate to } x_1 \} \subset K,\]
which is the stabilizer of the $K$-action on the conjugacy class of $x_1$. Indeed, there is a map 
\[ \rho\colon C_G(y_1) \to K \]
defined as follows. For $g\in C_G(y_1)$ pick a lift $\tilde{g} \in G^{\sc}$. Then 
\[\pi(\tilde{g} x_1 \tilde{g}^{-1}) = gy_1g^{-1} = y_1 = \pi(x_1), \]
so $\tilde{g} x_1 \tilde{g}^{-1} = k(g) x_1$ for a unique $k(g) \in K$, and we set $\rho(g)\coloneq k(g)$. This construction is well-defined by centrality of $K$, and one checks that it is a homomorphism. Moreover, its image is precisely $K(x_1)$. 

The kernel of $\rho$ consists of those $g \in C_G(y_1)$ which admit a lift that centralizes $x_1$, i.e.\ $\ker(\rho) = \pi(C_{G^{\sc}}(x_1))$. Centralizers in simply connected groups are connected, so $\ker(\rho)$ is connected and therefore $\ker(\rho) \subset C_G(y_1)^{\circ}$. On the other hand, the target of $\rho$ is finite, so actually $\ker(\rho) = C_G(y_1)^{\circ}$ and $\rho$ induces an isomorphism $\pi_0(C_G(y_1)) \cong K(x_1)$. So the translation action of $K$ on $\{G^{\sc}.x_{k} \mid k \in K \}$ makes this set a $K/K(x_1)$-torsor.

We claim that 
\[\cC \cap \Sigma = \{\pi(x_k) \mid k\in K \}. \]
To see this, note first that $x_k = zkx_1z^{-1}$ for some $z\in G^{\sc}$, so $\pi(x_k) = \pi(z)y_1 \pi(z)^{-1}$ and $\pi(x_k)\in \cC$. On the other hand, $x_k \in \Sigma^{\sc}$ implies $\pi(x_k) \in \Sigma$. Therefore $\pi(x_k) \in \cC \cap \Sigma$. 

Conversely, if $y \in \cC \cap \Sigma$, we can find $x \in \Sigma^{\sc}$ with $\pi(x) = y$. Moreover, $y$ is conjugate to $y_1$, so
\[\pi(x) = z y_1 z^{-1} = \pi(z') y_1 \pi(z')^{-1} = \pi(z' x_1 (z')^{-1} ) \]
for some $z \in G$ and $z' \in G^{\sc}$ a preimage of $z$. This means that
\[x = kz' x_1 (z')^{-1} =z' kx_1 (z')^{-1}\]
for some $k \in K$ and by definition, this implies $x = x_k$.

Since $\pi$ is injective on $\Sigma^{\sc}$, we see that $\cC \cap \Sigma$ is in bijection with the set of conjugacy classes in $G^{\sc}$ lying over $\cC$. In particular, it is a $K/K(x_1)$-torsor.

We describe the action of $T^{w}$ under this identification. For $t \in T^{w}$, pick a lift $\tilde{t} \in T^{\sc} \subset G^{\sc}$. Then $\pi(w^{-1}(\tilde{t}) \tilde{t}^{-1}) = 1$, since $w$ fixes $t \in T$. Therefore $w^{-1}(\tilde{t}) \tilde{t}^{-1} \in K$, and we obtain a homomorphism
\[\d \colon T^{w}\to K, t \mapsto \Ad_{w^{-1}}(\tilde{t}^{-1}) \tilde{t},\]
which is well-defined since two lifts of $t$ differ by a central element. 

Now we compute the action of $T^{w}$ on $\cC \cap \Sigma$. It is enough to do so on $y_1 = \pi(x_1)$. Write $y_1 = u_1 w \in Uw$. Then for $t \in T^{w}$ we have $ty_1t^{-1} = tu_1 t^{-1} w$. Let $\tilde{t}$ be a lift of $t$, and $\tilde{u}_{1}$ a lift of $u_{1}$. Then $x_1 = \tilde{u}_1w$ is the unique lift of $y_1$ in $\Sigma^{\sc}$, and $\tilde{t} \tilde{u}_{1} \tilde{t}^{-1} w \in \Sigma^{\sc}$ is the unique lift of $ty_{1}t^{-1}$ in $\Sigma^{\sc}$. Note that 
\[ \tilde{t} \tilde{u}_{1} \tilde{t}^{-1} w = \tilde{t} \tilde{u}_{1}w \delta(t)\tilde{t}^{-1} . \]
Since $\delta(t)$ is a central element, this means that $\tilde{t} \tilde{u}_{1} \tilde{t}^{-1} w \in \Sigma^{\sc}$ is conjugate to $\delta(t) x_1$. Therefore, the action of $T^{w}$ translates into multiplication by $\delta(t)$ under the identification $\cC \cap \Sigma \cong K/K(x_{1}) $ induced by the choice of $y_1 \in \cC \cap \Sigma$.

Consequently, in order to prove that the action is transitive, it suffices to show that $\delta\colon T^{w}\to K$ is surjective. Consider the exact sequence 
\[ 1 \to K \to T^{\sc} \to T \to 1. \]
Denote by $\G = \langle w \rangle$ the group generated by $w$. This is a cyclic group of order $h$. The above sequence induces a long exact sequence in group cohomology
\[1 \to K \to (T^{\sc})^w \to T^w \to H^1(\G, K) \to H^1(\G, T^{\sc}) \to ... \]
Here $\G$ acts trivially on $K$. We claim that $H^1(\G, T^{\sc}) = 0$. 

To see this, consider the exponential sequence
\[0 \to X_*(T^{\sc}) \to \frt^{\sc}\to T^{\sc} \to 1. \]
Since $\frt^{sc}$ is a vector space, $H^1(\G, \frt^{\sc}) = 0$, as taking invariants is exact for representations of finite groups. Using the long exact sequence associated to the exponential sequence, we conclude 
\[H^1(\G, T^{\sc}) = H^2(\G, X_*(T^{\sc})).\]
Moreover, $Q^{\vee} = X_*(T^{\sc})$ is the coroot lattice for $G$. For cyclic groups, a standard result tells us that
\[H^2(\G, Q^{\vee}) = (Q^{\vee})^{\G}/ NQ^{\vee} \]
where $N=1+w+...+w^{h-1} \in \ZZ[\G]$. Since $w$ is a Coxeter element, it is elliptic, so it acts without fixed points in the reflection representation. In particular, $(Q^{\vee})^{\G} = 0$, so $H^1(\G, T^{\sc}) = H^2(\G, Q^{\vee})=0$. 

We conclude that $T^w \to H^1(\G,K)$ is surjective. It is standard that this connecting homomorphism can be computed as follows (see \cite[Remark 1.21]{Milne}). For $t \in T^{w}$, pick a lift $\tilde{t} \in T^{\sc}$. Then the image of $t$ under the connecting homomorphism is the class of the cocycle
\[\G \to K, g \mapsto (g.\tilde{t}) \tilde{t}^{-1}.\]

Moreover, since $\G$ acts trivially on $K$, we can identify 
\[H^1(\G,K) = \Hom(\G,K).\]
Finally, since $\G$ is cyclic, we can pick the generator $w^{-1} \in \G$ to get a further identification
\[ \Hom(\G,K) \cong K[h] := \{k \in K \mid k^h =1 \},\]
by mapping $\phi \mapsto \phi(w^{-1})$. Under this isomorphism, the connecting homomorphism is identified with $T^{w} \to K[h], t \mapsto \delta(t^{-1})$. Since the connecting homomorphism is surjective onto $K[h]$, so is $\delta$.

To finish the proof, we note that for all simple groups $G$, the exponent of $K$ (i.e.\ the minimal number $k$ such that $x^k=1$ for all $x \in K$) divides the Coxeter number $h$. This implies that $K[h] = K$. 
Indeed, all possibilities for $K$ are subgroups of $Z(G^{\sc})$. For convenience, we list the possible centers in Table \ref{table:centers}, together with the Coxeter number $h$. One readily verifies that every possible exponent divides $h$.

\begin{table}[]
\begin{tabular}{l | l | l || l | l | l  }
\textup{Dynkin type} 		& 	$Z(G^{\sc})$  					&	 $h$						&  	\textup{Dynkin type} 	& 	$Z(G^{\sc})$ 	&  $h$			\\ 
\hline
$A_{n-1}$ 			&  	$\ZZ/n\ZZ$					& $n$	&  	$E_6$			&  	$\ZZ/3\ZZ$		&	$12$									\\
$B_{n}$				&  	$\ZZ/2\ZZ$					& $2n$	&  	$E_7$			&  	$\ZZ/2\ZZ$		&	$18$									\\
$C_{n}$				&  	$\ZZ/2\ZZ$					& $2n$	& 	$E_8$			&  	$1$		&	$30$									\\
$D_{n}$, $2\mid n$		&  	$\ZZ/2\ZZ\times \ZZ/2\ZZ	$					& $2(n-1)$&  	$F_4$			&    	$1$		&	$12$									\\
$D_{n}$, $2 \nmid n$	&  	$\ZZ/4\ZZ$		& $2(n-1)$&	$G_2$  			& 	$1$		&     $6$
\end{tabular}
\caption{Centers of simply connected simple groups}
\label{table:centers}
\end{table}

We have thus shown that $\delta\colon T^{w} \to K$ is surjective, and hence that the action of $T^{w}$ on $\cC \cap \Sigma$ is transitive.
\end{proof}

The generalized Steinberg section can be used to describe the braid space $\cM(\tilde{w})$ as follows. 

\begin{theorem}\label{thm:ellipticbraidstack}
    For an elliptic $w$ as above, there is an isomorphism of algebraic stacks
    \[\cM(\widetilde{w}) \cong [\Sigma_{w} / T^{w}]. \]
    
\end{theorem}

\begin{proof}

We first claim that there is an isomorphism of stacks 
\[[UwU / T^{w}U] \cong [BwB / B].\]

This can be seen as follows: Conjugation induces a map 
\[ B \times UwU \to BwB, (b,x) \mapsto bxb^{-1}. \]
This map factors through the quotient $B \times^{T^{w}U} UwU$, so we get a morphism
\[ B \times^{T^{w}U} UwU \to BwB, \]
which is equivariant for the left translation action of $B$ on $B \times^{T^{w}U} UwU$ and the conjugation action of $B$ on $BwB$.

First we check that $B \times UwU \to BwB$ is surjective. Indeed, when $w$ is elliptic, $T^{w}$ is finite. This means that the map
\begin{equation}
    T \to T, \quad t\mapsto t\Ad_{w}(t^{-1})
\end{equation}
is surjective, since it is a morphism between two tori of the same rank with finite kernel. 

Now given an element of $BwB = TUwU$, say $t'u_1wu_2$, we can find $t\in T$ such that $t' = t\Ad_{w}(t^{-1})$. So we have
\begin{equation}
    t'u_1wu_2 = t\Ad_{w}(t^{-1})u_1wu_2 = tu_{1}'wu_{2}'t^{-1}
\end{equation}
for $u_1',u_2'\in U$. Here we use the facts that $W$ normalizes $T$, and $T$ normalizes $U$. This proves that every element of $BwB$ is $B$-conjugate (in fact $T$ suffices) to an element of $UwU$, that is $t'u_1wu_2$ is the image of $(t,u'_1wu'_2)$.

To prove that the induced morphism $B \times^{T^{w}U} UwU \to BwB$ is an isomorphism, we use the following.

\begin{lemma}\label{lemma:fp}
    Let $w \in W$ (not necessarily elliptic). Let $u_1,u_2,v_1,v_2\in U$, and assume $(tu)(u_1wu_2)(tu)^{-1} = v_1wv_2$ for some $u\in U$. Then $t\in T^{w}$. 
\end{lemma}
\begin{proof}
    First we have 
    \begin{equation}
        v_1wv_2 = tuu_1wu_2u^{-1}t^{-1}. 
    \end{equation}
    You can move $t^{-1}$ to the left to get 
    \begin{equation}
        t\Ad_{w}(t^{-1}) u_1'wu_2' = v_1wv_2,
    \end{equation}
    for some $u_1', u_2' \in U$ (here we use that $T$ normalizes $U$, and $w$ normalizes $T$). So it suffices to prove that if $tu_1wu_2=v_1wv_2$, we need to have $t=1$. 

This is proved as follows. If $tu_1wu_2=v_1wv_2$, then 
\begin{equation}
    \tilde{u}=t\Ad_{t^{-1}}(v_{1}^{-1})u_1 = \Ad_{w}(v_2u_2^{-1}).
\end{equation}
Note that 
  $\Ad_{t^{-1}}(v_{1}^{-1})u_1 \in U$ and  $u_2'=\Ad_{w}(v_2u_2^{-1}) \in wUw^{-1}$. Now we have
  \begin{equation}
      t\Ad_{t^{-1}}(v_{1}^{-1})u_1 \in TU = B, 
  \end{equation}
  i.e.\ $\tilde{u} \in B$. However, $\tilde{u}$ is unipotent, so actually it lies in $U$. This means 
  \[t\Ad_{t^{-1}}(v_{1}^{-1})u_1 = \tilde{u} \in U,\]
so $t\in U$. This implies $t=1$. 
\end{proof}
This immediately implies $B \times^{T^{w}U} UwU \cong BwB$, and we get
\[[UwU / T^{w}U] \cong [BwB / B].\]
The second claim is that $[\Sigma_{w} / T^{w}]\cong [UwU / T^{w}U]$. For this consider the map
\[T^{w}U \times \Sigma_{w} \to UwU, (b,x )\mapsto bxb^{-1}, \]
which is clearly $T^{w}U$-equiviarant. We claim it induces an isomorphism 
\[T^{w}U \times^{T^{w}} \Sigma_{w} \to UwU. \]

This follows from the results of \cite{HL}: The map $T^{w}U \times \Sigma_{w} \to UwU$ is surjective, because $\Sigma_{w}$ intersects every $U$-orbit on $UwU$.
Furthermore, if $u_1w$ is conjugate to $u_2w$ by $tu \in T^{w}U$, we have 
\begin{equation}
    u_2w = tuu_1wu^{-1}t^{-1} = \Ad_{t}(u) \Ad_{t}(u_1) w \Ad_{t}(u^{-1}).
\end{equation}
It is clear that $\Ad_{t}(u_1) w \in \Sigma_{w}$. Hence, because $\Sigma_{w}$ intersects every orbit in a unique point, we see that $u_2 = \Ad_{t}(u_1)$. Moreover, since the action of $U$ on $UwU$ is free, we find that $u=1$. This proves the claim.
\end{proof}

\subsection{Physical rigidity of Kloosterman connections}
Recall that a $G$-connection $\cE$ on $\PP^1\setminus S$, $S$ finite, is called physically rigid if any other $G$-connection on $\PP^1\setminus S$ with isomorphic formal types at every $x \in S$ is isomorphic to $\cE$. 

\begin{theorem}\label{thm:KloostermanPhysicalRigidity}
Kloosterman $G$-connections are physically rigid. That is, for any regular conjugacy class $\cC$, and any isoclinic irregular class $\Theta_{h}$ of slope $1/h$, there is a unique $G$-connection on $\PP^1\setminus \{0,\infty\}$ which has irregular class $\Theta_{h}$ at $\infty$, and which is regular singular at $0$ and whose local monodromy at $0$ lies in the conjugacy class $\cC$.
\end{theorem}
Note that at first glance this seems stronger than physical rigidity, because we are not fixing the formal monodromy at $\infty$. However, being isoclinic of slope $1/h$ forces the formal monodromy to be uniquely determined (up to twisted conjugacy).
\begin{proof}
By Theorem \ref{thm:Comp}, a $G$-connection is completely determined by its associated Stokes-filtered $G$-local system. Moreover, by Lemma \ref{lem:isoclinicbraid}, the moduli stack of Stokes-filtered $G$-local systems with irregular class $\Theta_{h}$ at $\infty$ and which are regular singular at $0$ with local monodromy in the regular class $\cC$ is
\[\cM(\beta_{1/h}, \cC) = \mu^{-1}(\cC),\]
where $\mu\colon\cM(\beta_{1/h}) \to [G/\Ad(G)]$ is the monodromy map and $\beta_{1/h} = \tilde{w}$ for a Coxeter element $w$. In this situation, $\Sigma_{w}$ is Steinberg's cross-section, and by Theorem \ref{thm:ellipticbraidstack}, we have an isomorphism 
\[
\cM(\beta_{1/h}, \cC) \cong [(\Sigma_{w} \cap \cC)/T_{w}],
\]
Now Theorem \ref{thm:Twtrans} implies that $[(\Sigma_{w} \cap \cC)/T_{w}]$ has a unique $\CC$-point.
\end{proof}

\begin{remark}\label{rem:RigidityWCV}
    Once we know the equivalence between $G$-connections and Stokes $G$-local systems from Proposition~\ref{prop:EquivStokesGLS}, there should also be an approach to the rigidity of the above Kloosterman connections directly via the description of wild character varieties from \cite{BY}. Let us briefly explain this idea (suggested to us by Philip Boalch):

    The (twisted) wild character variety $\cM_\mathrm{B}(\cC)$ classifying Stokes $G$-local systems on $(\PP^1, \{0,\infty\})$ with the irregular type given by that of the Kloosterman connection and with the monodromy at $0$ in a fixed conjugacy class $\cC\subset G$ is explicitly described as follows (see \cite[\S 6]{BY}): After fixing a base point $b\in \partial_\infty$ on the circle at $\infty$, certain generators of the fundamental group of $\widetilde{\PP^1}$ and identifying $\Aut(\cF_b)\cong G$, we can write
    $$\cM_\mathrm{B}(\cC) =[ \{ (h,S_1,\ldots,S_r)\in H(\partial)\times \prod_{i=1}^r \SS \mathrm{to}_i \mid hS_r\cdots S_1 \in \cC \}/T] $$
    for the wild character stack. Here, $H(\partial)\subset G$ is the ``twist'' of $H=\mathrm{GrAut}(\cF_b)=T=(\CC^\times)^r\subset G$ determined by the irregular class. In this situation, one can identify $H(\partial)=Tw$ for a Coxeter element $w\in W$. Note that for this reason, no choice of twisted conjugacy class in $H(\partial)$ is necessary: There is only a single one, the reason again being that the map $T \to T, t \mapsto t\Ad_w(t^{-1})$ is surjective. 
    
    Moreover, if we number the singular directions in a positive sense starting at $b$, i.e.\ $\AA=\{d_1,\ldots,d_r\}\subset \partial_\infty$, then $\SS \mathrm{to}_i$ denotes the Stokes group at the singular direction $d_i$ (transported to $b$ in the natural way). One can then identify the product of Stokes groups with the space $U \cap \dot{w} U^{-} \dot{w}^{-1}$, similar to \cite{GH}, but without having to pass to a finite cover. The idea is to characterize the roots that contribute to the Stokes groups as follows. First pass to a finite cover, and then consider the roots that are positive for a half-period, but negative for a shifted half-period. This condition cuts out precisely $U \cap \dot{w} U^{-} \dot{w}^{-1}$ inside $U$ (understood to be defined by the first half-period).

This identifies the wild character stack  $\cM_\mathrm{B}(\cC)$ with the same space as above, i.e. 
    \[ \cM_\mathrm{B}(\cC) \cong [(\Sigma_{w} \cap \cC)/T_{w}]. \]
Hence $\cM_\mathrm{B}(\cC)$ consists of a single point, because $T_{w}$ acts transitively.

Finally, this can be understood as an instance where we identify wild character stacks explicitly with the quotient of a braid variety.
\end{remark}

\subsection{Generalized Kloosterman connections}
Here, we discuss a generalization of the above. As before, let $w$ be a regular elliptic element of minimal length in its conjugacy class. Let $m$ be its order, which is a regular elliptic number for $G$. Let $\b \defeq \tilde{w}$ and write $\cM(w) = \cM(\beta)$. 

In \cite{Lu}, a unipotent class is attached to each conjugacy class in $W$. Denote by $\cC_{w}$ the unipotent class attached to the conjugacy class of $w$. 

Recall the generalized Steinberg section $\Sigma_{w}$ constructed in \cite{HL}. He and Lusztig show that $T_w = \ker(t \mapsto t\Ad_w(t^{-1}))$ acts transitively on $\Sigma_w \cap \cC_{w}$. Note that $T_w$ is finite, since $w$ is elliptic. 

\begin{defn}
    A generalized Kloosterman connection of slope $1/m$ is an algebraic $G$-connection $(\cE,\nabla)$ on $\PP^1\setminus\{0,\infty\}$ which satisfies the following properties:
    \begin{enumerate}
        \item At $0$, the connection $\nabla$ has a regular singularity and the local monodromy lies in the unipotent class $\cC_{w}$. 
        \item At $\infty$, the connection is isoclinic of slope $1/m$. 
    \end{enumerate}
\end{defn}

Note that the space $\cM(w,\cC_{w})$ classifies Stokes filtrations of generalized Kloosterman connections.  

\begin{cor}\label{cor:genkloost}
Generalized Kloosterman connections are physically rigid. 
\end{cor}

The proof again uses Theorem \ref{thm:ellipticbraidstack}, to identify $\cM(w,\cC_{w}) = [(\Sigma_w \cap \cC_{w})/T_w]$, and the result then follows from the transitivity of the $T_w$-action. 

\subsection{Physical rigidity of Airy connections}
In this section, we study generalizations of the classical Airy equation. Recall that $h$ denotes the Coxeter number of $G$.
\begin{defn}
    An \textit{Airy connection} is an algebraic $G$-connection on $\AA^1$ which is isoclinic of slope $\nu = \frac{h+1}{h}$ at $\infty$. 
\end{defn}
Special cases were first constructed in \cite{KS2}, and general existence is part of \cite{JY}. Airy connections are cohomologically rigid. 
\begin{theorem}\label{thm:AiryPhysicalRigidity}
Airy $G$-connections are physically rigid. 
\end{theorem} 
As for Kloosterman connections, we will study the associated moduli space of Stokes-filtered $G$-local systems. The theorem will then follow from the following result. 
\begin{prop}\label{prop:Airyfil} Let $c \in W$ be a Coxeter element with lift $\tilde{c} \in \Br_W^+$. Recall that $\mu\colon \cM(\beta_\nu) \to [\Ad(G) \backslash G ]$ denotes the monodromy map. Define $\beta_\nu \defeq \tilde{c}^{h+1}$, and let $\cM(\beta_\nu,1) \defeq \mu^{-1}(1)$. Then $\cM(\beta_\nu,1) \cong [S\backslash \pt] = BS$ for some finite group $S$. 
\end{prop}

We will need to recall a few facts from \cite{CGG}. Let $\beta = \s_{i_1}\cdots\s_{i_r} \in \Br_{W}$ be any positive braid for now, and denote by $\delta = \Dem(\beta)$ its Demazure (or greedy) product. We denote by $s_{i_k}$ the simple reflection in $W$ which is the image of $\s_{i_k}$. In \textit{loc.~cit.}, the authors consider the braid variety
\begin{align*}
    X(\beta) = \{(B_1,\ldots,B_{r+1})\in (G/B)^{r+1} \mid B_1 = B, B_{i_k} \xrightarrow{s_{i_k}} B_{i_{k+1}}, B_{r+1} = \delta B_1 \}.
\end{align*}
It is a smooth irreducible affine variety of dimension $\ell(\beta) - \ell(\delta)$, which depends on $\beta$ only as a braid (that is, not on its decomposition). In addition, for a Weyl group element $w\in W$, we consider the slightly more general braid varieties 
\begin{align*}
    X(\beta, w) &= \{(B_1,\ldots,B_{r+1})\in (G/B)^{r+1} \mid B_1 = B, B_{i_k} \xrightarrow{s_{i_k}} B_{i_{k+1}}, B_{r+1} = w B_1 \},\\
    \tilde{X}(\beta, w) &= \{(B_1,\ldots,B_{r+1})\in (G/B)^{r+1} \mid B_{i_k} \xrightarrow{s_{i_k}} B_{i_{k+1}}, B_{r+1} \xrightarrow{w^{-1}} B_1 \}.
\end{align*}
There is an explicit description as an affine variety
\begin{align}\label{eqn:affbraid} X(\beta,w) = \{(z_1,\ldots,z_r) \in \AA^r \mid w^{-1} b_{i_1}(z_1) \cdots b_{i_r}(z_r) \in B \},
\end{align}
with $b_{i_k}(z_k) \in G$ satisfying 
\begin{align} \label{eqn:torusaction}
    t b_{i_k}(z_k) = b_{i_k}(\alpha_{i_k}(t) z_k) \Ad_{s_{i_k}}(t) 
\end{align}
for all $t\in T$. For $z=(z_1,\ldots,z_r)$ the corresponding point in $X(\beta,w)$ is given by 
\begin{align}\label{eqn:affcoords}
B_1= B, B_2 = b_{i_1}(z_1) B, \dots ,b_{i_1}(z_1) \dotsm b_{i_r}(z_r)B ).
\end{align}
There is a natural diagonal action of $T$ on $X(\beta,w)$ induced from the action of $T$ on $G/B$. Equations \eqref{eqn:torusaction} and \eqref{eqn:affcoords} give an explicit description in terms of affine coordinates that we do not spell out here. In type A the description can be found in \cite[\S 2.2]{CGGS}. The following is a generalization (or variant) of Lemma 2.10 in \textit{loc.~cit}. We denote by $\pi\colon \Br_{W} \to W$ the natural quotient, whose kernel is the subgroup of pure braids.

\begin{lemma}\label{lem:finitestabs}
    Let $\beta$ be a positive braid, $w\in W$, and assume $\pi(\beta)w^{-1}\in W$ is an elliptic element. Then the action of $T$ on $X(\beta, w)$ has finite stabilizers. 
\end{lemma}
\begin{proof}
    Let $z \in X(\beta,w)$ and assume $t\in T$ stabilizes it. Write $w_1=\pi(\beta)$ and $b_{\beta}(z)=b_{i_1}(z_1) \dotsm b_{i_r}(z_r)$. Then, by \eqref{eqn:torusaction}, we have
    \[t b_{\b}(z) = b_{\b}(z) \Ad_{w_1^{-1}}(t).\]
    Moreover, $\Ad_{w^{-1}}(t)w^{-1}b_{\b}(z) = w^{-1} b_{\b}(z) \Ad_{w_1^{-1}}(t)$, 
implying 
\begin{align*} \Ad_{w^{-1}}(t)w^{-1}b_{\b}(z)\Ad_{w_1^{-1}}(t)^{-1} &= w^{-1} b_{\b}(z). 
\end{align*}
By \eqref{eqn:affbraid}, we know that $w^{-1}b_{\b}(z) \in B$, so we can write $w^{-1}b_{\b}(z) = du$ with $d\in T$ and $u\in U$. This implies 
\begin{align*} du &= \Ad_{w^{-1}}(t)du\Ad_{w_1^{-1}}(t)^{-1}  \\
&=\Ad_{w^{-1}}(t)d \Ad_{w_1^{-1}}(t)^{-1}\Ad_{w_1^{-1}}(t) u \Ad_{w_1^{-1}}(t)^{-1},
\end{align*}
where $\Ad_{w^{-1}}(t)d \Ad_{w_1^{-1}}(t)^{-1} \in T$ and 
$\Ad_{w_1^{-1}}(t) u \Ad_{w_1^{-1}}(t)^{-1}\in U$. Therefore, we end up with the equality 
\begin{align*} \Ad_{w^{-1}}(t)d \Ad_{w_1^{-1}}(t)^{-1} =d
\end{align*}
in $T$, hence $\Ad_{w_1 w^{-1}}(t) = t$. This means the stabilizer of $z$ is contained in
\[T^{\pi(\beta)w^{-1}} = \{t'\in T \mid \Ad_{w_1 w^{-1}}(t') = t'\}.\]
Since $w_1w^{-1} = \pi(\beta)w^{-1}$ is elliptic, $\frt^{\pi(\beta)w^{-1}}=0$, and $T^{\pi(\beta)w^{-1}}$ is finite. 
\end{proof}

Denote by $\Delta$ the braid lift of the longest element $w_0\in W$. It is a braid word for the braid called the \textit{half-twist}. Recall that for any positive braid $\beta$, we denote by $\Dem(\beta)$ its Demazure product. 
\begin{proof}[Proof of Proposition \ref{prop:Airyfil}]
Recall that one has $\cM(\beta_\nu ,1) = [G\backslash \cM^{\sharp}(\beta_\nu ,1)]$, where
$\cM^{\sharp}(\beta_\nu ,1) = \tilde{X}(\beta_\nu,1)$.
By \cite[Definition 6.12]{Be}, the element $\tilde{c}^{h}$ represents the full-twist in the braid group, i.e.\ $\tilde{c}^{h} = \Delta^2$. In other words, $\b_{\nu} = \tilde{c}\Delta^2$, and we let $\beta'=\tilde{c}\Delta$. By \cite[\S 3.2]{CGG}, we have an identification $\tilde{X}(\beta_{\nu},1) = \tilde{X}(\beta'\Delta,1) = \tilde{X}(\beta',w_0)$, compatible with the $G$-action. Recall that $\tilde{X}(\beta',w_0)$ classifies tuples of Borel subgroups which are in relative position given by $\beta'$ and such that the first and last Borel subgroups (say $B_1$ and $B_{r'+1}$) in the tuple are opposite. We can therefore use the $G$-action to move $(B_1,B_{r'+1})$ to $(B,B^{\opp})$. This identifies
\[ [G\backslash \cM^{\sharp}(\beta_\nu ,1)] \cong [T \backslash X(\beta',w_0)],\]
because $T$ is the stabilizer of $(B,B^{\opp})$. The fact that $\beta'$ has $\Delta$ as a right factor implies $\Dem(\beta') = w_0$. Therefore $X(\beta',w_0) = X(\beta')$ is smooth, irreducible and has dimension 
\[\ell(\beta')-\ell(w_0) = \ell(\beta)-2\ell(w_0) = (h+1)\dim(T)-2\dim(U) = \dim(T). \]
Moreover, $\pi(\beta')w_0 = \pi(\beta) = c$, which is elliptic, and by Lemma \ref{lem:finitestabs} the action of $T$ on $X(\beta')$ has finite stabilizer. As a consequence, the $T$-action is transitive. Indeed, pick any $z\in X(\beta')$ with a closed orbit (this always exists) and let $S\subset T$ be its stabilizer. Then $\dim(T.z)= \dim(T) - \dim(S) = \dim(T) = \dim X(\beta')$, and we necessarily have $T.z = X(\beta')$. This shows $\cM(\beta_\nu ,1) \cong [T\backslash X(\beta')] = [S \backslash \pt]$. 
\end{proof}

\end{document}